\documentclass{amsart}

\usepackage{amsmath, amsfonts, amsthm, amssymb}
\usepackage{epsfig}
\usepackage{graphics}
\usepackage{color}

\renewcommand{\geq}{\geqslant}
\renewcommand{\leq}{\leqslant}

\newtheorem{theorem}{Theorem}[section]
\newtheorem{lemma}[theorem]{Lemma}
\newtheorem{proposition}[theorem]{Proposition}
\newtheorem{remark}[theorem]{Remark}

\newtheorem*{main-theorem}{Main Theorem}

\newtheorem*{remark*}{Remark}

\numberwithin{equation}{section}

\def\phi{\varphi}

\newcommand{\mc}{\mathcal}
\newcommand{\rr}{\mathbb{R}}
\newcommand{\nn}{\mathbb{N}}
\newcommand{\cc}{\mathbb{C}}
\newcommand{\hh}{\mathbb{H}}
\newcommand{\zz}{\mathbb{Z}}

\newcommand{\eps}{\epsilon}

\newcommand{\pl}{\partial}
\newcommand{\x}{\times}

\newcommand{\til}{\widetilde}

\newcommand{\cjd}{\rangle}
\newcommand{\cjg}{\langle}

\newcommand{\demi}{\frac{1}{2}}

\newcommand{\la}{\lambda}

\newcommand{\indic}{\operatorname{1\negthinspace l}}

\def\phi{\varphi}

\def\be{\begin{eqnarray*}}
\def\ee{\end{eqnarray*}}
\def\ben{\begin{eqnarray}}
\def\een{\end{eqnarray}}
\def\ker{\text{ker}}

\def\L2R{L_{\text{Rest}}^2}

\def\11{\mathds{1}}

\def\HS{\mathbb{H}}

\def\L2c{L^2_{\text{comp}}}

\def\<{\langle}
\def\>{\rangle}

\begin{document}

\title[Random walk on cusps]{
Random walk on surfaces with hyperbolic cusps}

\author[H. Christianson]{Hans~Christianson}
\address{Massachusetts Institute of Technology, Department of Mathematics\\
77 Mass. Ave., Cambridge, MA 02139-4307, USA}
\email{hans@math.mit.edu}
\author{Colin~Guillarmou}
\address{DMA, U.M.R. 8553 CNRS\\
Ecole Normale Sup\'erieure\\
45 rue d'Ulm\\ 
F 75230 Paris cedex 05 \\France}
\email{cguillar@dma.ens.fr}

\author[L. Michel]{Laurent~Michel}
\address{Lab. Dieudonn\'e\\
Univ. de Nice Sophia-Antipolis \\
Parc Valrose\\
06108 Nice\\ 
FRANCE}
\email{lmichel@math.unice.fr}

\maketitle
\begin{abstract}
We consider the operator associated to a random walk on finite volume surfaces with hyperbolic cusps. We study the spectral gap (upper and lower bound) associated to this operator and deduce some rate of convergence of the iterated 
kernel towards its stationary distribution. 
\end{abstract}

\section{Introduction}

In this work, we study the operator of random walk 
on finite volume surfaces with hyperbolic cusps.  On a Riemannian manifold $(M,g)$ with finite volume, 
the $h$-random walk operator is simply defined by averaging functions on geodesic balls of size $h>0$ as follows
\[ K_hf(m):=\frac{1}{|B_h(m)|}\int_{B_h(m)}f(m'){dv}_g(m')\]
where $B_h(x):=\{m'\in M; d(m',m)\leq h\}$ is the geodesic ball of center $m\in M$ and radius $h$, and $d(.,.), |B_h(m)|$ 
denote respectively the Riemannian distance and the Riemannian volume of $B_h(m)$.
This operator appeared in the recent work of Lebeau-Michel \cite{LebMi}, 
in which they study the random walk operator on compact manifolds.  They studied in particular the spectrum of 
this operator for small step $h>0$, and proved a sharp spectral gap for $K_h$, which provides
the exponential rate of convergence of the kernel $K_h^N(m,m')dv_g(m')$ of the iterated operator 
to a stationary probability mesure, in total variation norms. 
Related works on Metropolis algorithm were studied in \cite{DiaLeb} on the real line and \cite{DiaLebMi} in higher dimension. 
All these results rely on a very precise analysis of the spectrum of these operators (localization of eigenvalues, Weyl type estimates, eigenfunction estimates in $L^\infty$ norm). For an overview of this subject and more references on convergences of Markov kernels, we refer to \cite{Dia}.

More recently, the two last authors studied such random walk operators on unbounded domain of the flat Euclidian space endowed with a smooth probability density \cite{GuMi}. In this situation and for certain densities, since the domain is unbounded, the random walk operator may have essential spectrum at distance $O(h^2)$ from $1$ and the uniform total variation estimate fails to be true.

The motivation of the present work is to consider the simplest case of non-compact manifolds 
for which the kernel 
\[k_h(m,m')= \frac 1 {|B_h(m)|}\indic_{d(m,m')<h}dv_g(m')\] 
is still a Markov kernel, and which have a radius of injectivity equal to $0$.
The non-compactness of the manifold should create some essential spectrum for $K_h$ and 
it is not clear a priori that a spectral gap even exists in that case. Intuitively, the walk could need infinitely many steps 
to fill the whole manifold and approach the stationary measure in a total variation norm.
For surfaces, the radius of injectivity tending to $0$ at infinity 
makes the geometric structure of balls near infinity more complicated, and they will typically change topology from 
something simply connected to some domains with non trivial $\pi_1$.  It is then of interest to study what  
types of result one can or cannot expect in this setting. \\

Let us now be more precise.
Consider a surface $(M,g)$ with finite volume and finitely many ends $E_0,\dots E_n$,
with $E_i$ isometric to a hyperbolic cusp
\[ (t_i,\infty)_t \x (\rr/\ell \zz)_z \textrm{ with metric } g=dt^2+e^{-2t}dz^2.\]
for some $t_i>0$. Each end can also be viewed as a subset of the quotient $\cjg\gamma\cjd\backslash\hh^2$
of $\hh^2$ by an abelian group generated by one translation $\gamma:(x,y)\in\hh^2\to(x,y+\ell)\in\hh^2$ where the hyperbolic plane is 
represented by $\hh^2=\{ (x,y)\in \rr_+\x\rr\}$. 
 
We denote by $B_h(m)$ the geodesic ball in $M$ of radius $h>0$ and center $m$, then $|B_h(m)|$ will denote its volume with respect to
 $g$. Let $d\nu_h$ be the probability measure on $M$ defined by $d\nu_h=\frac{|B_h(m)|}{Z_h}{\rm dv}_g(m)$, 
where $Z_h$ is  a renormalizing constant. We define the random walk operator $K_h$ by 
 \[K_hf(m):=\frac{1}{|B_h(m)|}\int_{B_h(m)}f(m')dv_g(m')\]
Then, $K_h$ maps $L^\infty(M,d\nu_h)$ into itself, $L^1(M,d\nu_h)$ into itself, both with norm $1$. Hence, it maps $L^2(M,d\nu_h)$ into itself with norm $1$. Moreover, it is self-adjoint on $L^2(M,d\nu_h)$. Hence, the probability density $d\nu_h$ is stationary for $K_h$, that is $K_h^t(d\nu_h)=d\nu_h$ for any $x\in M$, where 
$K_h^t$ denotes the transpose operator of $K_h$ acting on Borel measures. 
In that situation, it is standard that the iterated kernel $K_h^n(x,dy)$ converges to the stationary measure $d\nu_h$ when $n$ goes to infinity. The associated rate of convergence is closely related to the spectrum of $K_h$ and more precisely to the distance between $1$ and the largest eigenvalue less than $1$. The main result of this paper is the following

\begin{theorem}\label{th:main} 
There exists $h_0>0$ and $\delta>0$ such that the following hold true:
\begin{enumerate}
\item[i)] For any $h\in]0,h_0]$, the essential spectrum of $K_h$ acting on $L^2(M,d\nu_h)$ is given by the interval
\[
I_h=[\frac h{\sinh(h)}A,\frac h{\sinh(h)}]
\]
where $A=\min_{x>0}\frac{\sin(x)}{x}>-1$.
\item[ii)]  For any $h\in ]0,h_0]$, $Spec(K_h)\cap [-1,-1+\delta]=\emptyset$.

\item[iii)] There exists $c>0$ such that for any $h\in ]0,h_0]$, $1$ is a simple eigenvalue of $K_h$ and   the spectral gap $g(h):={\rm dist}({\rm Spec}(K_h)\setminus\{1\},1)$  enjoys
\[
ch^2\leq g(h)\leq \min\Big(\frac{(\la_1+\alpha(h))h^2}{8}, 1-\frac{h}{\sinh(h)}\Big)
\]
where $\la_1$ is the smallest non-zero $L^2$ eigenvalue of $\Delta_g$ on $M$ and $\alpha(h)$ a function
tending to $0$ as $h\to 0$.
\end{enumerate}

\end{theorem}

Compared to the results of \cite{LebMi} in the compact setting, our result is weaker since we are not able to provide a 
localization of the discrete spectrum of $K_h$ in terms of the Laplacian spectrum. This
is due to the fact that in the cusp, the form of the geodesic balls of radius $h$ changes dramatically 
and, in some sense, the approximation of $K_h$ by a function of the Laplacian is not 
correct anymore in this region of the surface.\\

This paper is organized as follows.
In the next section we describe the form of the operator in the cusp part of the manifold. In section 3, we study the essential spectrum of $K_h$ acting on $L^2(M,d\nu_h)$.
In section 4, we prove part ii) of the above theorem and we start the proof of iii).
The upper bound on the gap is shown by computing the operator $K_h$ on
smooth functions (in fact on the eigenfunctions of the Laplace
operator). The left lower bound  is obtained by showing a Poincar\'e inequality:
\[
\<(1-K_h) f,f\>_{L^2(d\nu_h)}\geq C h^{2}(\|f\|^2_{L^2(d\nu_h)}-\<f,1\>_{L^2(d\nu_h)}^2).
\]
For the proof of this inequality, we study separately the compact region of the manifold and the cusp. The cusp study is detailed in section 4.

In section 5, we construct some quasimodes for $K_h$ (namely the eigenfunctions of the Laplace operator). This permits to exhibit some eigenvalues of $K_h$ close to $1$ and to give a sharp upper bound on the spectral gap.
In section 6, we use the previous results to study the convergence of $K_h^n(x,dy)$ towards $d\nu_h$. We prove that the difference between these two probabilities is of order 
$C(x)e^{-ng(h)}$ in total variation norm and that the constant $C(x)$ can not be chosen uniformly with respect to $x$ (contrary to the case of a compact manifold).

In the last section, we give some smoothness results on the eigenfunctions of $K_h$. This should be the first step towards a more precise study of the spectrum in the spirit of \cite{LebMi}.

Finally, we observe that it will be clear from the proofs that we only need to consider the case
with a unique end $E:=E_0$ for $M$, and so we shall actually assume that there is only one end to simplify exposition.\\

\textbf{Ackowledgement}. H.C. is partially supported by NSF grant
DMS-0900524.  C.G. and L.M. are partially supported by ANR grant ANR-09-JCJC-0099-01. 
H.C and C.G. would like to thank MSRI (and the organizers of the program 'Analysis on singular spaces') where part of the work was done in Fall 2008.

\section{Geometry of balls and expressions of the random walk operator}

\subsection{Geometry of geodesic balls in the cusp}

In this section we study geodesic balls in the cusp.  First we briefly recall
what balls look like in the hyperbolic space $\HS^2 = \{ (x,y) \in
\rr_+ \times \rr\}$ with the same metric $(dx^2+dy^2)/x^2$.  
It is convenient to use coordinates $x = e^t$, in which case the volume element becomes
\[{dv}_{g} = e^{-t} dt dy.\]
A ball $B((e^t,y),r)$ centered at $(e^t,y)$ and of radius $r$ in $\HS^2$ is a
Euclidean ball centered at $(e^t \cosh r,y)$ and of Euclidean radius $e^t \sinh r$.
That is, a ball of radius $r$ and center $e^t$ has its ``top'' at
$(e^{t + r},y)$ and its ``bottom'' at $(e^{t-r},y)$ .  By changing to polar coordinates, it is easy to see
that a ball in $\HS^2$ has volume 
\[|B(( e^t,y), r) | = 2\pi \int_0^r \sinh(r') d r'=2\pi (\cosh(r)-1).\]

\begin{figure}
\hfill
\centerline{\input{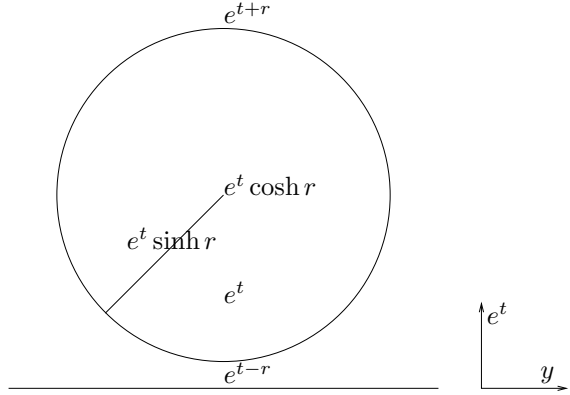}}
\caption{\label{fig:ball1} The hyperbolic ball in Euclidean
  coordinates.  The center in hyperbolic coordinates is at height
  $e^t$, and in Euclidean coordinates at $e^t \cosh r$.}
\hfill
\end{figure}

The cusp end $E_0$ of $M$ is identified with the region $x>x_1$ inside
$\cjg \gamma\cjd \backslash\hh^2$, where $\gamma(x,y)=(x,y+\ell)$ and $x_1>0$ is a fixed number.
A fundamental domain of the cyclic group $\cjg \gamma\cjd$ in 
$\hh^2$ is given by the strip $S:=\{x>0,\ell\geq y>0\}$. The end $E_0$
can thus be seen as the quotient $\cjg \gamma\cjd\backslash (S\cap \{x>x_1\})$.
The geodesic ball $B_h(m)$ in the cusp  end $E_0$ can be obtained by considering 
\[ B_h(m)=\pi(\{ m'\in\hh^2; d_{\hh^2}(m,m')\leq h\})\]
if we view $m$ as being in $S$, and where $\pi:\hh^2\to \cjg \gamma\cjd\backslash \hh^2$ is the canonical projection of the covering. 

As a consequence, we see that, as long as the Euclidean radius of
$B_h(m)$ is less than or equal to
the width $\ell$ of the strip $S$, then $B_h(m)$ can be considered as a ball of radius $h$
in $\hh^2$, while when the Euclidean radius is greater than or equal to
$\ell$, i.e. when $t\geq \log(\ell/2)-\log(\sinh(h))$, then the ball
overlaps on itself  
and can be respresented in $S$ by 
\begin{equation}\label{ballh}
 B_h(m)=\bigcup_{j=-1}^1 \{ (x',y')\in S; |e^t\cosh(h)-x'|^2+|y+j\ell-y'|^2\leq e^{2t}\sinh(h)^2\} 
 \end{equation}
if $m=(e^t,y)\in S$ and there are at most two of these three regions which have non-empty interior. \\
In particular, if $(x=e^t, y=\ell/2)$, then the ball $B_h(m)$ is given by the region 
\[ \{ 0\leq y'\leq \ell ;   |x'-e^{t}\cosh(h)|\leq
\sqrt{e^{2t}\sinh^2(h)-|y'-\ell/2|^2}  \}.\] 
See Figures \ref{fig:ball1}, \ref{fig:ball2}, and \ref{fig:ball3}.



\begin{figure}
\hfill
\begin{minipage}[t]{.48\textwidth}
\centerline{\input{ball2}}
\caption{\label{fig:ball2} The hyperbolic ball of radius $r$ is tangent to itself
 when the center is at $t = \log(\ell/2 \sinh (r))$.  For
 $t>\log(\ell/2 \sinh (r)$ the ball overlaps on itself.}
\end{minipage}
\hfill
\begin{minipage}[t]{.48\textwidth}
\centerline{
\input{ball3}}
\caption{\label{fig:ball3} The hyperbolic ball of radius $r$ for $t > \log(\ell/2
 \sinh (r))$ with shifted center.}
\end{minipage}
\hfill
\end{figure}

We are now in position to give a couple of explicit expressions for $K_h$ which will be used later.
\subsection{First expression of $K_h$ in the cusp}\label{firstexp}
Let us use the coordinates $(t,y)$ in the strip $S$ defined above so that $E_0:=\cjg \gamma\cjd \backslash S=\{(e^t,y)\in(x_0,\infty)\x (\rr/\ell\zz)\}$, 
for some $x_0>0$.
The first expression is obtained by integrating the function on vertical lines covering the geodesic ball.
Let us denote by $B_h(t,y)$ the geodesic ball on $E_0$ centered at $(e^t,y)$ of radius $h$.
It is easily seen that 
the operator $K_h$ acting on a function $\psi(t,y)$ with support in the cusp $E_0$ can be written in the form 
\begin{equation}\label{Kh}
\begin{gathered}
K_h\psi(t,y)=\frac{\indic_{[\log(\ell/\sinh(h)),\infty)}(t)}{|B_h(t,y)|}\int_{y-\ell/2}^{y+\ell/2}
\int_{t+t_-(e^{-t}|y-y'|)}^{t+t_+(e^{-t}|y-y'|)}\til{\psi}(t',y')e^{-t'}dt'dy'\\
+ \frac{\indic_{(0,\log(\ell/\sinh(h)))}(t)}{|B_h(t,y)|}
\int_{\sqrt{|\cosh(h)e^t-e^{t'}|^2+|y-y'|^2}<e^t\sinh(h)}\til{\psi}(t',y')e^{-t'}dt'dy',
\end{gathered}
\end{equation}
where $\til{\psi}$ is the lift of $\psi$ to the covering $\hh^2\to
\cjg z\to z+\ell\cjd\backslash\hh^2$, and 
\[t_\pm(z)=\log(\cosh(h)\pm\sqrt{\sinh(h)^2-|z|^2}).\]
We write $K_h\psi$ as a sum of two parts: $K_h\psi  = K_h^1\psi+K^2_h\psi$ where $K_h^1\psi$ is supported in $\{t\geq\log(\ell/2\sinh(h))\}$ 
and $K_h^2\psi$ in $\{t\leq\log(\ell/2\sinh(h))\}$.
The action of $K^1_h$ on $\psi$ can be written, using change of variables, 
\begin{equation}\label{Kh1}
K^1_h\psi(t,y)=\frac{1}{|B_h(t,y)|}\int_{-\frac{e^{-t}\ell}{2}}^{\frac{e^{-t}\ell}{2}}
\int_{\log(\cosh(h)-\sqrt{\sinh(h)^2-z^2})}^{\log(\cosh(h)+\sqrt{\sinh(h)^2-z^2})}\til{\psi}(t+T,y+ze^t)e^{-T}dTdz.
\end{equation}
Decomposing $\til{\psi}$ in Fourier series in $y$, one can write 
$\til{\psi}(t,y)=\sum_{k=-\infty}^\infty e^{\frac{2\pi iky}{\ell}}a_{k}(t)$ and  
one has 
\begin{equation}\label{expressionK^h_1}
\begin{gathered}
K^1_h\psi_1(t,y)=\sum_{k=-\infty}^\infty e^{\frac{2\pi iky}{\ell}} K^1_{h,k}a_k(t)\\
K^1_{h,k}a_k(t):=\frac{1}{|B_h(t,y)|}\int_{-\frac{e^{-t}\ell}{2}}^{\frac{e^{-t}\ell}{2}}e^{\frac{2\pi ikze^t}{\ell}}
\int_{\log(\cosh(h)-\sqrt{\sinh(h)^2-z^2})}^{\log(\cosh(h)+\sqrt{\sinh(h)^2-z^2})}a_k(t+T)e^{-T}dTdz.
\end{gathered}
\end{equation}
Using Plancherel theorem and computing the Fourier transform of $e^{-T}\indic_{[t_-,t_+]}(T)$, we obtain  
\begin{equation}\label{plancherel}
\begin{gathered}
\int_{t_-(z)}^{t_+(z)}a_k(t+T)e^{-T}dT=\int e^{it\xi}\hat{a}_k(\xi)\sigma(z,\xi)d\xi\\
\sigma(z,\xi):=\frac{(\cosh(h)+\sqrt{\sinh(h)^2-z^2})^{1+i\xi}-
(\cosh(h)-\sqrt{\sinh(h)^2-z^2})^{1+i\xi}}{(1+i\xi)(1+z^2)^{1+i\xi}}.
\end{gathered}
\end{equation} 
Therefore, $|B_h|K^1_{h,k}$ corresponds to a pseudo-differential operator on $\rr$ with symbol 
\[\sigma_k(t,\xi):=\int_{-\frac{e^{-t}\ell}{2}}^{\frac{e^{-t}\ell}{2}}e^{\frac{2\pi ikze^t}{\ell}}
\sigma(z,\xi)dz.\] 
The operator $K_h^{2}$ can be written in the same way 
\begin{equation}\label{Kh2}
K^2_h\psi(t,y)=\frac{1}{4\pi(\sinh(\frac{h}{2}))^2}\int_{-\sinh(h)}^{\sinh(h)}
\int_{t_-(z)}^{t_+(z)}\til{\psi}(t+T,y+ze^t)e^{-T}dTdz.
\end{equation}

\begin{remark}\label{remarkvolume}
Note that,  taking $\psi=1$ in \eqref{Kh1}, one has 
\begin{equation}\label{eq:formule_volume}
|B_h(t,y)|=\int_{-\frac{e^{-t}\ell}{2}}^{\frac{e^{-t}\ell}{2}}\frac{2\sqrt{\sinh(h)^2-z^2}}{1+z^2}dz.
\end{equation}
For $t>\log(\ell/2\sinh(h))$, we thus obtain the estimate 
\begin{equation}\label{asymptvolBh}
|B_h(t,y)|=2\ell \sinh(h)e^{-t} + O(e^{-3t}/\sinh(h))=|R_h(t,y)|+O(e^{-3t}/\sinh(h))
\end{equation}
where $|R_h(t,y)|$ denotes the volume of $R_h(t,y):=\{ (e^{t'},y')\in S; |t'-t|<h\}$, 
which is the `smallest' cylinder of the cusp containing $B_h(t,y)$.

On the other hand, there exists $C>0$ such that for all $t\geq \log(\ell/2\sinh(h))$, $|B_h(t)|\geq Che^{-t}$.
\end{remark}

\subsection{Second expression of $K_h$ in the cusp}\label{secondexp}

We give another expression of $K_{h}$ by integrating along horizontal lines instead.  Writing as above 
\[u(t,y)=\sum_{k}e^{\frac{2i\pi ky}{\ell}}u_k(t)\]
when $u$ is supported in an exact cusp $\{t>T\}$, 
the operator $K_h$ can be decomposed as a direct written near this region by 
\[K_hu(t,y)=\sum_k e^{\frac{2i\pi ky}{\ell}}K_{h,k}u_k(t).\]
Let us define the following 
\[T_\pm(t):=\cosh(h)\pm\sqrt{\sinh^2(h)-e^{-2t}\ell^2/4}, \quad \alpha(T):=\frac{2}{\ell}\sqrt{\sinh^2(h)-(\cosh(h)-e^T)^2}\]
then an easy computation by integrating on horizontal lines $t'={\rm cst}$ in the cusp gives 
that the operator $K_{h,k}$ decomposes into $K_{h,k}^{j}$ for $j=1,2,3$ where 
\begin{equation}
\begin{gathered}\label{decompKhk}
 K^1_{h,k}u(t)=\frac{\ell}{2|B_h|}\int_{-h}^{\log
  T_-(t)}\int_{-\alpha(T)}^{\alpha(T)}u(t+T)e^{i\pi kze^t}e^{-T}dzdT, \\
 K_{h,k}^2u(t)=\frac{\ell}{2|B_h|}\int_{\log
  T_+(t)}^h\int_{-\alpha(T)}^{\alpha(T)}u(t+T)e^{i\pi
  kze^t}e^{-T}dzdT, \\
 K_{h,k}^3u(t)=\frac{\ell}{2|B_h|}\int_{\log T_-(t)}^{\log
  T_+(t)}\int_{-1}^{1}u(t+T)e^{i\pi kz}e^{-T-t}dzdT 
\end{gathered}
\end{equation}
when $e^t\sinh(h)\geq \ell/2$ while 
\begin{equation}\label{souslecusp}
K_{h,k}u(t)=|B_h|^{-1}\frac{\ell}{2}\int_{-h}^h\int_{-\alpha(T)}^{\alpha(T)}u(t+T)e^{i\pi kze^t}e^{-T}dzdT
\end{equation} 
when $e^t\sinh(h)\leq \ell/2$.
Suppose first $e^t\sinh(h)\geq \ell/2$, then when $k\not=0$ the terms $K_{h,k}^{j}$ can be simplified by integrating in $z$ to 
\begin{equation}\label{knot=0}
\begin{gathered}
(K_{h,k}^1+K_{h,k}^2)u(t)=\frac{\ell}{|B_h|}\int_{-h}^{\log T_-(t)}+\int_{\log T_+(t)}^h u(t+T)
\frac{\sin(k\pi e^t\alpha(T))}{\pi ke^{t}\alpha(T)}e^{-T}\alpha(T)dT,\\
K^3_{h,k}u(t)=0 
\end{gathered}
\end{equation} 
while if $k=0$, 
\begin{equation}\label{k=0}
\begin{gathered}
(K^1_{h,0}+K_{h,0}^2)u(t)=|B_h|^{-1}\ell\int_{-h}^{\log T_-(t)}+\int_{\log(T_+(t))}^h u(t+T)\alpha(T)e^{-T}dT\\ 
 K_{h,0}^3u(t)=|B_h|^{-1}\ell\int_{\log T_-(t)}^{\log T_+(t)}u(t+T)e^{-T-t}dT.
\end{gathered}
\end{equation}
The obvious similar expression  holds when $e^t\sinh(h)\leq \ell/2$.

\section{Essential spectrum of $K_h$ on $L^2(M)$}

Recall that $K_h$ is a self-adjoint bounded operator on $L^2(M,d\nu_h)$, with norm equal to $1$. Moreover, $1\in {\rm Spec}(K_h)$. In this section we show that the essential spectrum of $K_h$ is well separated from $1$.

\begin{theorem}\label{essspec}
The essential spectrum of $K_h$ acting on $L^2(M,d\nu_h)$ is given by the interval 
\[I_h:=\left[\frac{h}{\sinh(h)}A,\frac{h}{\sinh(h)}\right].\]
with $A:=\min_{x>0}\frac{\sin(x)}{x}$.
\end{theorem}
\begin{proof} The operator $K_h$ acting on $L^2(M,d\nu_h)$ is unitarily equivalent to the operator 
\[\til{K}_h: f\to \til{K}_hf(m):=\frac{1}{|B_h(m)|^{\demi}}\int_{B_h(m)}f(m')\frac{1}{|B_h(m')|^\demi}dv_g(m') \]
acting on $L^2(dv_g)$.  
Now, using $(t,y)$ variables in the cusp, let us take $t_0\gg 0$ arbitrarily large and let 
$\chi_{t_0}(t,y):=1-\indic_{[t_0,\infty)}(t)$ which is compactly supported.
Clearly, from the fact that $K_h$ propagates supports at distance at most $h$, 
we can write 
\[\til{K_h}=\indic_{[t_0,\infty)}\til{K}_h\indic_{[t_0,\infty)}+\chi_{t_0}\til{K}_h\chi_{t_0}+\chi_{t_0}\til{K}_h\indic_{[t_0,t_0+h]}+\indic_{[t_0,t_0+h]}\til{K}_h\chi_{t_0}.\] 
Since $\chi_{t_0},\chi_{t_0\pm h}$ are compactly supported, it is obvious that the integral kernel of the last three operators is   
in $L^2(M\x M; dv_g\otimes dv_g)$ and so these operators are Hilbert-Schmidt and thus compact. 
Now by a standard theorem, the essential spectrum of $\til{K}_h$ is then the essential spectrum of $\indic_{[t_0,\infty)}\til{K}_h\indic_{[t_0,\infty)}$ 
for all large $t_0\gg 0$. Let us consider the operator $T_h$ on $L^2(M,dv_g)$ defined by
\[T_hu(t,y)=\frac{1}{|R_h(t)|^\demi}\indic_{[t_0,\infty)}(t)\int_{y-\frac\ell 2}^{y+\frac \ell 2}\int_{t-h}^{t+h}\indic_{[t_0,\infty)}(t')\frac{u(t',y')}{|R_h(t')|^\demi}e^{-t'}dt'dy'.\] 
where $|R_h(t)|=2\ell e^{-t}\sinh(h)$ is the measure of the rectangle $t'\in [t-h,t+h]$ as in Remark \ref{remarkvolume}.
If $e^{t_0}$ is chosen much bigger than $h^{-1}$, we have from Remark \ref{remarkvolume} that 
$|B_h(t)|=|R_h(t)|(1+O(h^{-2}e^{-2t}))$ which implies from Schur's Lemma that the operator 
$T_h-(1-\chi_{t_0})\til{K}_h(1-\chi_{t_0})$ has $L^2$ norm bounded by  $Ch^{-2}e^{-2t_0}$ for some $C>0$.
Therefore this norm can be made as small as we like by letting $t_0\to \infty$ and  
it remains to study the essential spectrum of $T_h$ when $t_0$ is chosen very large. 
Remark that $T_h$ can be decomposed in Fourier modes in the $S^1$ variable $y$ like we did for $K_h$ in the cusp, 
and only the component corresponding to the constant eigenfunction of $S^1$ is non-vanishing.  
Therefore the norm of $T_h$ is bounded by the norm of the following operator acting on $L^2(\rr,e^{-t}dt)$ 
\[ f \to u(t)=\frac{\indic_{[t_0,\infty)}(t)}{2\sinh(h)e^{-t/2}}\int_{t-h}^{t+h}\indic_{[t_0,\infty)}(t')f(t')e^{-t'/2}dt'.\]
or equivalently 
\[f \to u(t)=\frac{\indic_{[t_0,\infty)}(t)}{2\sinh(h)}\int_{t-h}^{t+h}\indic_{[t_0,\infty)}(t')f(t')dt'\]
acting on $L^2(\rr,dt)$. This can also be written as a composition $\indic_{[t_0,\infty)}A_h\indic_{[t_0,\infty)}$ where $A_h$ is the operator which is a  Fourier multiplier on $\rr$
\[A_h= \mc{F}^{-1}\frac{\sin(h\xi)}{\sinh(h)\xi}\mc{F}.\]  
>From the spectral theorem, it is clear that this operator has only continuous spectrum and its
spectrum is given by the range of the smooth function $\xi\to \sin(h\xi)/\sinh(h)\xi$, i.e. by
$I_h$, and its operator norm is $h/\sinh(h)$. Suppose now that $\lambda\in {\rm Spec}_{\rm ess}(\til{K}_h)$ then $\lambda$ belongs to the  spectrum of $\indic_{[t_0,\infty[}\til{K}_h\indic_{[t_0,\infty[}$ for all $t_0$. If the spectrum of $\indic_{[t_0,\infty[}A_h\indic_{[t_0,\infty[}$ is included in $I_h$, then  letting $t_0\to \infty$ 
implies that $\la\in I_h$, by the norm estimate on the difference of the two operators.
Since 
\[   \frac{h}{\sinh(h)}A||f||^2_{L^2}\leq \cjg A_h\indic_{[t_0,\infty[}f,\indic_{[t_0,\infty[}f\cjd \leq \frac{h}{\sinh(h)} 
||f||^2_{L^2},\]
the spectrum of $\indic_{[t_0,\infty[}A_h\indic_{[t_0,\infty[}$ is included in $I_h$, we just have to prove the other inclusion. 
To prove it is exactly $I_h$, we have to construct Weyl sequences for $\til{K}_h$. 
Consider the orthonormalized sequence $(u_n)_{n\in\nn}$ of $L^2$ orthonormalized functions 
\[u_n(t):= 2^{-n/2}e^{i\la t}\indic_{[2^n,2^{n+1}]}(t) ,\quad n\in \nn\] 
then a straightforward computation shows that 
\[ \left|\left|\left(\indic_{[2^n-1,\infty)}A_h\indic_{[2^n-1,\infty)}
-\frac{\sin(\la h)}{\la \sinh(h)}\right)u_n\right|\right|_{L^2(\rr,dt)}=O(2^{-n/2}).\]
But also $\til{K}_hu_n=\indic_{[2^{n}-1,\infty)}\til{K}_h(\indic_{[2^{n}-1,\infty)}u_n)$ and thus by taking $n$ large  and using the norm estimate on $\indic_{[2^{n}-1,\infty)}\til{K}_h\indic_{[2^{n}-1,\infty)}-T_h$ with $t_0:=2^{n}-1$ in the definition of $T_h$, we obtain
\[\left|\left|\left(\til{K}_h-\frac{\sin(\la h)}{\la \sinh(h)}\right)u_n\right|\right|_{L^2}\leq C(2^{-n/2}+ h^{-2}e^{-2^{n+1}})\] 
and letting $n\to \infty$, we can apply the Weyl criterion to deduce that $I_h$ is the essential spectrum of $\til{K}_h$. 
\end{proof}

\section{Spectral gap of order $h^2$ for $K_h$ on $L^2$}

In this section, we show the existence of a spectral gap of order $h^2$ for 
$K_h$ on acting on $L^2(M,d\nu_h)$. Recall that $d\nu_h(m)=\frac{|B_h(m)|}{Z_h}dv_g$ where
$Z_h$ is a positive constant such that this $d\nu_h$ is a probability measure. In particular, 
in our case $h^2/C<Z_h<Ch^2$ for some $C>0$. 

Let us first show that the bottom of the spectrum of $K_h$ is
uniformly bounded away from $-1$.
\begin{proposition}\label{bottom}
There exists $\delta>0$, $h_0>0$ such that for all $0<h\leq h_0$
\begin{equation}\label{trouen-1}
{\rm Spec}(K_h)\cap[-1,-1+\delta]=\emptyset.
\end{equation}
\end{proposition}
\begin{proof}
This amounts to prove an estimate of  the form 
\[\mc{G}_h(f)\geq \delta ||f||^2_{L^2(M,d\nu_h)}\]
where 
\[\mc{G}_h(f)=\cjg (1+K_h)f,f\cjd_{L^2(M,d\nu_h)}=
\frac{2}{Z_h}\int_{d(m,m')\leq h}(f(m)+f(m'))^2 dv_g(m)dv_g(m).\]
We proceed as in \cite{DiaLebMi} and consider a covering $\cup_j\omega_j=M$ of $M$ by geodesic
balls of diameter $h/4$ and such that for any $j$, the number of $k$ such that $\omega_j\cap\omega_k\neq\emptyset$ 
is less than $N$ for some $N$ independent of $h$. Then, using that the volume of $|B_h(m)|$
is constant of order $h^2$ when $t(m)\in [t_0,\log(\ell/2\sinh(h))]$ (for some $t_0>0$ independent of $h$), 
we deduce easily that ${\rm Vol}_g(\omega_j)>C\max_{m\in\omega_j}|B_h(m)|$ 
for some uniform $C>0$ when $\omega_j$ has center 
in $\{t\leq \log(2/\ell\sinh(h))\}$, 
while when it has center $m_j$ such that $t(m_j)\geq  \log(2/\ell\sinh(h))$, we have 
${\rm Vol}_g(\omega_j)\geq Ce^{-t_j}h\geq C'\max_{m\in \omega_j} |B_h(m)|$ 
for some $C,C'>0$ uniform in $h$, by using \eqref{asymptvolBh} . As a consequence, we obtain
\[\begin{split}
\mc{G}_h(f)&\geq\frac{1}{2NZ_h} \sum_j\int_{\omega_j\times\omega_j,d(m,m')<h}(f(m)+f(m'))^2dv_g(m)dv_g(m')\\
&\geq\frac{1}{2NZ_h} \sum_j\int_{\omega_j\times\omega_j}((f(m)+f(m'))^2dv_g(m)dv_g(m')\\
&\geq\frac{1}{NZ_h} \sum_j{\rm Vol}_g(\omega_j)\int_{\omega_j}|f(m)|^2dv_g(m)\\
\mc{G}_h(f)&\geq \frac{C}{N} \int_{M}|f(m)|^2\frac{|B_h(m)|}{Z_h}dv_g(m)=\frac{C}{N}||f||^2_{L^2(M,d\nu_h)}
\end{split}\]
and this achieves the proof.
\end{proof}

Let us define the following functionals on $L^2(M,d\nu_h)$
\begin{equation}\label{defVf}
V_h(f):=||f||_{L^2(M,d\nu_h)}^2-\cjg f,1\cjd_{L^2(M,d\nu_h)}=\demi\int_{M\x M}(f(m)-f(m'))^2d\nu_h(m)d\nu_h(m')
\end{equation}
\begin{equation}\label{defEf}
\mc{E}_h(f):=\cjg(1-K_h)f,f\cjd_{L^2(M,d\nu_h)}=\frac{1}{2Z_h}\int_{d(m,m')<h}(f(m)-f(m'))^2
dv_g(m)dv_g(m').
\end{equation}

The spectral gap $g(h)$ can be defined as the largest constant such that 
\[
V_h(f)\leq \frac 1{g(h)}\mc{E}_h(f) , \quad \forall f\in L^2(M,d\nu_h)
\]
with the convention $g(h)=\infty$ if $1$ has multiplicity greater than $1$.\\

For the convenience of the reader, let us first give a brief summary of the method we are going to use to 
obtain a lower bound on $g(h)$: we will split the surface into two surfaces with boundary, one of which is 
compact (call it $M_0$), the other being an exact cusp (call it $E_0$), 
then we shall double them along their respective boundary to obtain $X=M_0\sqcup M_0$ and $W=E_0\sqcup E_0$,
and extend smoothly the metric $g$ from $M_0$ to $X$ and from $E_0$ to $W$ in such a fashion that $W$ is a surface of 
revolution $\rr\x (\rr/\ell\zz)$ with two isometric cusps near infinity. We will reduce the problem of 
getting a lower bound on  $g(h)$ to that of obtaining a lower bound on the spectral gap of both
random walk operators on $X$ and $W$. The compact case $X$ has been studied in \cite{LebMi}, and the main
difficulty will be to analyze $W$, which will be done in the next section. 
To that aim, we will use Fourier decomposition in the $\rr/\ell\zz$ variable and show 
that only the $0$-Fourier mode plays a serious role, 
then we will reduce the analysis of the operator acting on the $0$-Fourier mode to the analysis 
of a random walk operator with an exponentially decaying measure density on the real line, 
which is a particular case of the setting studied in \cite{GuMi}.\\

Let us now prove the
\begin{theorem}\label{gaph^2}
There exists $C<1/6$ and $h_0>0$ such that for any $h\in]0,h_0]$ 
\[
Ch^2\leq g(h)\leq 1-\frac{h}{\sinh(h)} =\frac{h^2}{6}+O(h^4).
\]
In particular, $1$ is a simple eigenvalue of $K_h$.
\end{theorem}
\begin{proof}
The upper estimate on $g(h)$ is a corollary of Theorem \ref{essspec} (using the Weyl sequences in the proof). 
Let us then study the lower bound, which is more involved.
The surface $M$ decomposes into a disjoint union $M=M_0\cup E_0$ with $M_0$ compact and $E_0$ isometric to the cusp 
$\simeq \{(t,y)\in (t_0-1,\infty)\x \rr/\ell \zz \} $ with metric
$dt^2+e^{-2t}dy^2$ (see Figure \ref{fig:manifold}).
In particular, the regions $M_0$ is compact with diameter independent of $h$.
Let us extend the function $m\mapsto t(m)$ smoothly to the whole
surface $M$ so that $0<t(m)<t_0-1$ for all $m\in M_0$.

\begin{figure}
\centerline{
\input{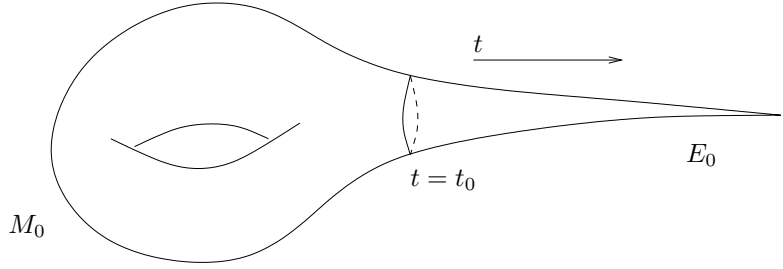}}
\caption{\label{fig:manifold} The decomposition $M = M_0 \cup E_0$.}
\hfill
\end{figure}

We then decompose the functional $V_h(f)$ according to this splitting of $M$ and we define for $
0 \leq a<c<b\leq \infty$
\[V_h^{[a,b]}(f):=\demi\int_{t(m)\in[a,b],t(m')\in[a,b]}(f(m)-f(m'))^2d\nu_h(m)d\nu_h(m'),\] 
\[I_h^{c}(f)=\demi \int_{t(m)\in [a,c],t(m')\in[c,b]}(f(m)-f(m'))^2d\nu_h(m)d\nu_h(m').\]
One then has for $c\in (a,b)$
\begin{equation}\label{Vhf}
V_h^{[a,b]}(f)=V_h^{[a,c]}(f)+V_h^{[c,b]}(f)+2I_h^{c}(f).
\end{equation}
Let us deal with the interaction term : 
$I_h^{c}(f)=\frac{1}{\nu_h(C_{c})}\int_{s\in C_{c}} I_h^{c}(f) d\nu_h(s)$ 
where $C_{c}:=\{m\in M;  c-1<t(m)<c+1\}$ and thus 
\[I_h^{c}(f)\leq 2\int_{s\in C_{c}}\int_{\substack{t(m)\in[a,c],\\
t(m')\in[c,b]}}(f(m)-f(s))^2+
(f(s)-f(m'))^2 d\nu_h(m)d\nu_h(m')\frac{d\nu_h(s)}{\nu_h(C_{c})}\] 
which implies for $a+1\leq c\leq b-1$,
\[I_h^{c}(f)\leq \frac{2\nu_h(t(m)\in[c,b])}{\nu_h(C_{c})}V_h^{[a,c+1]}(f)+
\frac{2\nu_h(t(m)\in [a,c])}{\nu_h(C_{c})}V_h^{[c-1,b]}(f).\]
Assume now that $c$ satisfies $e^ch\leq C$ for some $C>0$ independent
of $h$.   
Since the measure $c_0\leq d\nu_h/dv_g\leq c'_0 $ in $\{t<\log(\ell/2\sinh(h))\}$ for some $c_0,c'_0>0$ and 
$ c_1 e^{-t}/h<d\nu_h/dv_g<c_2 e^{-t}/h$ in $\{t>\log(\ell/2\sinh(h))\}$ for some $c_1,c_2>0$, we immediately 
deduce (using also \eqref{Vhf}) that there exists $C>0$ such that for all $f\in C_0^\infty(M)$ and $h$ small
\[V_h^{[a,b]}(f)\leq C\Big(V_h^{[a,c+1]}(f)+e^{c-a}V_{h}^{[c-1,b]}(f)\Big).\]
Using this estimate with $c=t_0$ (which is independent on $h$), we obtain   
\begin{equation}\label{Vhsplit}
V_h(f)\leq \,  C\Big(V_{h}^{[0,t_0]}(f)+e^{t_0}V_h^{[t_0,\infty]}(f)\Big)
\end{equation}
We also notice the inequality 
\begin{equation}\label{Ehcontrol}
\mc{E}_h(f)\geq \frac{1}{4}\Big(
\mc{E}_{h}^{[0,t_0+1]}(f)+\mc{E}_h^{[t_0-1,\infty]}(f)\Big), 
\end{equation}
where, for any $a,b\in [0,\infty]$, 
\[\mc{E}_h^{[a,b]}(f):= \frac{1}{2Z_h}\int_{t(m'),t(m)\in[a,b], d(m,m')<h}(f(m)-f(m'))^2dv_g(m)dv_g(m').\]
Using the preceding observations, it  remains to prove the inequalities 
\begin{equation}\label{toprove}
\mc{E}^h_{[0,t_0]}(f)\geq Ch^2 V^h_{[0,t_0]}(f),\quad  \mc{E}^h_{[t_0-1,\infty]}(f)\geq Ch^2 V^h_{[t_0-1,\infty]}(f).
\end{equation}
where we have used the fact that $e^{t_0}$ is independant of $h$.

Let us prove the following Lemma, which will deal with the non-compact region.  
\begin{lemma}\label{t0-1}
For any $f\in L^2(M)$, the following inequality holds 
\[\mc{E}_h^{[t_0-1,\infty]}(f)\geq Ch^2 V_h{[t_0-1,\infty]}(f).\]
\end{lemma}
\begin{proof}
We are going to prove 
\[
\begin{split}
\frac 1 {Z_h}\int_{m,m'\in E_0,d(m,m')<h}&(f(m)-f(m'))^2dv_g(m)dv_g(m')\\
&\geq Ch^2\int_{m,m'\in E_0}(f(m)-f(m'))^2d\nu_h(m)d\nu_h(m').
\end{split}
\]
Recall that $E_0=[t_0-1,\infty[\times \rr/\ell\zz$ is endowed with the metric $g=dt^2+e^{-2t}dy^2$. Without loss of generality, we can assume that $t_0=1$. Let us consider the surface $W:=\rr_t\x (\rr/\ell\zz)_y$, and view $E_0$ as the subset $t>0$ of $W$.
We equip $W$ with a warped product metric extending $g$ (and then still denoted $g$) to $t\leq 0$ as follows: $g:=dt^2+e^{-2\mu(t)}dy^2$ where $\mu(t)$ is a smooth function on $\rr$ which is equal to $|t|$ in $\{t>0\}\cup \{t<-1\}$ and 
such that $e^{-\mu(t)}\geq c_0 e^{-t}$ in $t\in[-1,0]$ for some
constant $c_0>0$ (see Figure \ref{fig:cusp-double}). 
 As a consequence, there exists some constant $C>0$ such that 
\begin{equation}\label{eq:controle_densite_sym}
\forall t\in\rr,\,\frac 1 C e^{-\mu(t)}\leq e^{-\mu(-t)}\leq Ce^{-\mu(t)}
\end{equation}
We denote by $d(m,m')$ the distance for the metric $g$ on $W$, $dv_{g}$ the volume form, $|B_{h}(m)|=v_{g}(B(m,h))$ the volume of the geodesic ball of radius $h$ and center $m$ associated to this metric $g$ on $W$. 
Consider also the probability measure $d\nu^W_{h}=\frac{|B_{h}(m)|}
{Z^W_{h}}dv_{g}(m)$, where $Z^W_{h}\in [h^2/C,Ch^2]$ (for some $C>1$)
is a renormalizing constant.

\begin{figure}
\centerline{
\input{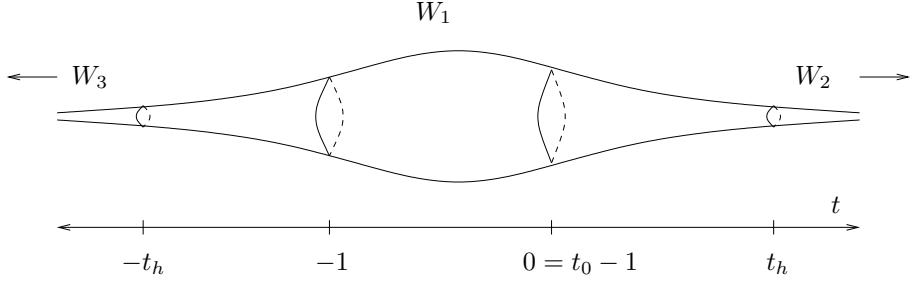}}
\caption{\label{fig:cusp-double} The surface of revolution $W$, which
  is a doubling of the cusp region $E_0 = \{ t \geq t_0 -1 =0 \}$ in
  these coordinates.  For later applications in Section \ref{S:gap-revolution}, we write $W = W_1 \cup
  W_2 \cup W_3$, with $W_2, W_3$ the regions where $| t | \geq t_h =
  \log ( \ell / 2 \sinh (h)) -1$.}
\hfill
\end{figure}

For $g\in L^2(E_0)$, let us define
\[
\mc{E}_h^W(g):=\frac 1 {Z_{h}^W}\int_{m,m'\in W,d(m,m')<h}(g(m)-g(m'))^2dv_{g}(m)dv_{g}(m')
\]
\[
V_h^W(g):=\int_{m,m'\in W}(g(m)-g(m'))^2 d\nu_{h}^W(m)d\nu_{h}^W(m').
\]
Any function $f\in L^2(E_0)$ can be extended to a function $f^s\in L^2(W)$, symmetric with respect to the involution $t\to -t$. 
Splitting $W\times W$ in four regions, we have
\be\label{4regions}
\begin{split}
\mc{E}_h^W(f^s)=&\frac{ Z_h}{Z_{h}^W}\mc{E}_h^{[0,\infty)}(f)
+2\int_{\substack{t(m)>0,t(m')<0, \\
d(m,m')\leq h}}(f^s(m)-f^s(m'))^2 dv_{g}(m)dv_{g}(m)\\
&+\int_{\substack{t(m)<0,t(m')<0,\\
d(m,m')\leq h}}(f^s(m)-f^s(m'))^2dv_{g}(m)dv_{g}(m)
\end{split}
\ee
We denote $\sigma: W\to W$ the involution $\sigma(t,y):=(-t,y)$ and use the change of variables $m\mapsto \sigma(m), m'\to \sigma(m')$ in the last term, and $m'\to \sigma(m')$ in the second term.
Using the assumptions on the metric $g$, we observe the following inclusions
\[\begin{gathered}
\{(m,m')\in W\x W; t(m)>0, t(m')>0, d(\sigma(m),\sigma(m'))\leq h\} \\
 \subset 
\{(m,m')\in W\x W; t(m)>0 , t(m')>0, d(m,m')\leq 2h\}, \,\, \textrm{ and } \end{gathered}\] 
\[\begin{gathered}
\{(m,m')\in W\x W; t(m)>0, t(m')>0, d(m,\sigma(m'))\leq h\} \\
\subset 
\{(m,m')\in W\x W; t(m)>0 , t(m')>0, d(m,m')\leq 2h\}.\end{gathered}\] 
The first inclusion comes from $e^{-\mu(t)}\geq e^{-\vert t\vert}/2$, while the second 
follows simply from $d(m,m')\leq d(m,\sigma(m'))+d(m',\sigma(m'))$
and the fact that $d(m',\sigma(m'))=2t(m')\leq h$ if $d(m,\sigma(m'))\leq h$.
Combined with \eqref{eq:controle_densite_sym} and the fact that $c\leq Z_h/Z_h^W\leq 1/c$ for some $0<c<1$, 
we see that the terms in the right hand side of \eqref{4regions} are bounded above by $C\mc{E}_{2h}^{[0,\infty)}(f)$ for some $C$, and we then deduce that for all small $h>0$
\begin{equation}\label{eh/2}
\mc{E}_{\frac{h}{2}}^W(f^s)\leq C\mc{E}_h^{[0,\infty)}(f).
\end{equation}
The proof of the following proposition is deferred to the next section.
\begin{proposition}\label{casrevol}
There exists $C>0$ and $h_0>0$ such that for all $f\in L^2(W)$ and all $h\in]0,h_0]$, we have:
\[Ch^2V_h^W(f)\leq \mc{E}_h^W(f)\]
\end{proposition}
Combining this Proposition with \eqref{eh/2} and the inequality  
$V_h(f)\leq V_h^W(f^s)\leq CV_{\frac{h}{2}}^W(f^s)$ which is a consequence of 
$d\nu_{h}^W/d\nu_{\frac{h}{2}}^W<\sqrt{C}$ for some $C>0$, we have proved Lemma \ref{t0-1}.
\end{proof}

We now analyze the compact regions which have diameter bounded uniformly with respect to $h$, i.e. $M_0$.
\begin{lemma}\label{cheminM0}
There exists $C$ independent of $h$ such that for all $f\in C_0^\infty(M)$
\[\mc{E}_h^{[0,t_0]}(f)\geq Ch^{2}V_h^{[0,t_0]}(f).\]
\end{lemma}
\begin{proof}
We shall use the same arguments as for the non-compact part, which is to reduce the problem to a closed
compact surface which doubles $M_0$.
We start by defining the surface $X:=M_0\sqcup M_0$ obtained by doubling $M_0$ along the circle $t=t_0$, and we equip it with a smooth structure extending that of $M_0$ and with a metric extending $g$, which we thus still denote $g$. We shall assume that $g$ has the form $g=dt^2+e^{-2\mu(t)}dy^2$ in a small open collar neighbourhood of 
$\{t=t_0\}$ (with size independent of $h$), where $\mu(t)$ is a function extending $t$ to a neighbourhood 
$t_0-\eps\leq  t\leq t_0+\eps$ of $\{t=t_0\}$ with $e^{-\mu(t)}\geq
c_0 e^{-t}$, $c_0 >0$.
Now repeating the same arguments as those of the proof of Lemma \ref{t0-1}, we see that it suffices to show 
that 
\[ \cjg (1-K^X_{h})f,f\cjd_{L^2(X,d\nu_{h}^X)}\geq Ch^2 (||f||^2_{L^2(X,d\nu_{h}^X)}-\cjg f,1\cjd^2_{L^2(X,d\nu_{h}^X)})\]
for any $f\in L^2(X)$, where $K^X_{h}$ is the random walk operator on $X$ for the metric $g$, defined just like 
for $M$, and 
$d\nu_{h}^X(m):={\rm Vol}(\{m\in X; d_g(m,m')\leq h\})dv_g/Z_{h,X}$
 for some normalizing constant $Z_{h}^X>0$ so that $d\nu_{h}^X$ is a probablity measure.
Now this estimates follows from the main Theorem of Lebeau-Michel \cite{LebMi}, where 
they show a spectral gap of order $h^2$ for the random walk operator $K^X_{h}$ on any compact manifolds $(X,g)$.
\end{proof}

The proof of the Theorem is thus achieved, provided we have shown  
Proposition \ref{casrevol}, i.e. the spectral gap on the surface of revolution $W$.
\end{proof}

\section{Spectral gap for the random walk on a surface of revolution}
\label{S:gap-revolution}
In this section, we consider the surface of revolution $W=\rr_t\x (\rr/\ell\zz)_y$ equipped with a metric 
$g=dt^2+e^{-2\mu(t)}dy^2$ where $\mu$ is a function equal to $|t|$ in $|t|\geq t_0$ for some fixed $t_0$ (a priori not necessarily the $t_0$ of previous Sections).
This can be considered as the quotient $\cjg y\to y+\ell\cjd \backslash \rr^2$ of $\rr^2$ equipped with the metric 
$dt^2+e^{-2\mu(t)}dy^2$ by a cyclic group $G$ of isometries generated by one horizontal translation.
We shall consider the random walk operator $K^W_h$ on $W$, defined as usual by 
\[K^W_hf(m)=\frac{1}{|B_h(m)|}\int_{B_h(m)}f(m')dv_{g}(m')\] 
where $B_h(m)$ denotes the geodesic ball of center $m$ and radius $h$ and $|B_h(m)|$ its volume for the 
measure $dv_{g}$. We assume that $h$ is small enough so that the ball $B_h(m)$ is diffeomorphic to a 
Euclidean ball of radius $h$ in $|t|\leq 2$. 

To simplify notations we will drop the superscripts $W$ referring to $W$, 
noting that we just have to remember we are working on the surface of revolution $W$ in this Section.\\

The Dirichlet form and the variance associated to this operator are defined as usual by
$\mc{E}_h(f)=\<(1-K_h)f,f\>_{L^2(W,d\nu_h)}$ and 
$V_h(f)=\|f\|^2_{L^2(W,d\nu_h)}-\<f,1\>_{L^2(W,d\nu_h)}^2$, where $d\nu_h(m)$ denotes the probability measure 
$\frac{|B_h(m)|}{Z_h}dv_{g}(m)$ for a certain renormalizing constant $Z_h$.

The main result of this section is the following 
\begin{proposition}\label{revolution}
There exists $C>0$ and $h_0>0$ such that for all $f\in L^2(W)$ and all $h\in]0,h_0]$, we have:
\begin{equation}
\label{EhVh}
Ch^2V_h(f)\leq \mc{E}_h(f).
\end{equation}
\end{proposition}
\begin{proof}
The expression of the operator acting on functions supported in $|t|>t_0+1$ is given in subsection \ref{secondexp}, 
since it corresponds to the random walk operator on a hyperbolic cusp. In particular, the operator $K_h$ preserves 
the Fourier decomposition in the $\rr/\ell\zz$ variable when acting on functions supported in $\{|t|>t_0+1\}$. 

Let us then study its form when acting on functions supported in $|t|\leq t_0+2$.
For any $v\in\rr$, the translation $y\to y+v$ on $\rr^2=\rr_t\x\rr_y$ descends to an isometry of $(W,g)$, and 
thus the geodesic ball $B_h(t,y)$ on $W$ has the same volume as $B_h(t,y')$ for all $y,y'\in \rr/\ell\zz$, i.e. the volume
$|B_h(t,y)|$ is a function of $t$, which we will denote $|B_h(t)|$ instead.

As long as $h$ is smaller than the radius of injectivity at $(t,y)$ (i.e. when $t<\log(\ell/2\sinh(h))$), 
the ball $B_h(t,y)$ is included in a fundamental domain 
of the group $G$ centered at $y$, i.e. a vertical strip $|y'-y|<\ell$ of width $\ell$, and $B_h(t,y)$ 
corresponds to a geodesic ball of center $(t,y)$ and radius $h$ in $\rr^2$ for the metric $dt^2+e^{-2\mu(t)}dy^2$. 
The reflection $(t,y')\to (t,2y-y')$ with fixed line $y'=y$ is an isometry of the metric $dt^2+e^{-2\mu(t)}dy^2$ on $\rr^2$, and thus $d((t,y),(t',y')=d((t,y),(t',2y-y'))$
where $d$ is the distance of the metric $g$. In particular, the ball $B_h(t,y)$ is symmetric with respect 
to the line $y'=y$. It can thus be parameterized by 
\[ B_h(t,y):=\{ (t',y'); |t-t'|\leq h, |y-y'|\leq \alpha_h(t,t')\}\]
for a certain continous function $\alpha_h(t,t')$ which satisfies $\alpha_h(t,t-h)=\alpha_h(t,t+h)=0$ (this corresponds the bottom and top of the ball) and $\alpha_h(t,t)=he^{-\mu(t)}$ (this corresponds to the `middle' of the ball).
It is easily seen that  $\alpha_h(t,t')\geq \eps h$ for some $\eps>0$ if $|t'-t|\leq h/2$.
Let us now check that $K_h$ preserves the Fourier decomposition in the
$y$ variable. Here we first suppose that $f\in L^2$ is supported in $|t|\leq t_0+2$. Then 
$f=\sum_k f_k(t)e^{2i\pi ky/\ell}$ for some $f_k(t)\in L^2(\rr,e^{-\mu(t)}dt)$, and we have  
\begin{equation}\label{formuleKmu}
\begin{split}
K_hf(t,y)= & \sum_{k\in\zz} \frac{1}{|B_h(t)|}\int_{t-h}^{t+h}\int_{y-\alpha_h(t,t')}^{y+\alpha_h(t,t')} 
f_k(t')e^{2i\pi ky'/\ell}e^{-\mu(t')}dy'dt'\\
=& \sum_{k\not=0} e^{2i\pi ky/\ell}\frac{2}{|B_h(t)|}\int_{t-h}^{t+h} 
f_k(t')\frac{\sin(2\pi k\alpha_h(t,t')/\ell)}{2\pi k\alpha_h(t,t')/\ell}\alpha_h(t,t')e^{-\mu(t')}dt'\\
& +\frac{2}{|B_h(t)|}\int_{t-h}^{t+h} \alpha_h(t,t')
f_0(t')e^{-\mu(t')}dt'  \\
K_{h} f(t,y)=:& \sum_{k\in\zz} (K_{h,k}f_k)(t)e^{2i\pi ky/\ell}.
\end{split}
\end{equation}
Notice in particular that 
\begin{equation}\label{formulaball}
|B_h(t)|=\int_{t-h}^{t+h} 2\alpha(t,t')e^{-\mu(t')}dt'. 
\end{equation}
Moreover, combining with the computations in subsection \ref{secondexp}, the expression \eqref{formuleKmu} and \eqref{formulaball}
can be extended to the whole surface $W$ by setting  
\begin{equation}\label{formulealphah}
\alpha_h(t,t')=\min\Big(e^{t}\sqrt{\sinh(h)^2-(\cosh(h)-e^{t'-t})^2} , \ell/2\Big)
\end{equation}
when $t\geq t_0+1$.

We start by proving the statement on the non-zero Fourier modes in $\rr/\ell\zz$.
\begin{lemma}\label{contraction}
There exists $\eps>0,h_0>0$ such that for all $k\not=0$, all $0<h\leq h_0$ and $f\in L^\infty(\rr)$
\[\| K_{h,k} f \|_{L^\infty} \leq (1-\eps h^2) \| f
\|_{L^\infty} \]
and for all $f \in L^2(\rr,|B_h(t)|e^{-\mu(t)}dt)$ the following $L^2$
estimate holds true: 
\begin{equation}\label{contraction2}
||K_{h,k}f||_{L^2(\rr,�|B_h(t)|e^{-\mu(t)}dt)}\leq (1-\eps h^2) ||f||_{L^2(\rr, |B_h(t)|e^{-\mu(t)}dt)}.
\end{equation}
Finally, there exists $\eps>0,h_0>0$ such that for all $0<h\leq h_0$, all $k\not=0$, all 
$\tau>t_0$ and all $f\in L^2(\rr,|B_h(t)|e^{-\mu(t)}dt)$ supported
in $|t|\geq \tau$, we have
\begin{equation}\label{contraction3}
||K_{h,k}f||_{L^2(\rr,�|B_h(t)|e^{-\mu(t)}dt)}\leq (1-\eps \min(k^2e^{2\tau}h^2,1)) ||f||_{L^2(\rr, |B_h(t)|e^{-\mu(t)}dt)}.
\end{equation}
\end{lemma}
\begin{proof}
The proof uses the expression for $K_{h,k}$ given in the equations \eqref{formuleKmu}, with $\alpha_h(t,t')$ given by \eqref{formulealphah} in $\{|t|\geq t_0+1\}$.
If $f\in L^\infty(\rr)$, one easily has from \eqref{formuleKmu} 
\begin{equation}\label{normLinfty}
 ||K_{h,k}f||_{L^\infty}\leq ||f||_{L^\infty} \sup_{t}\Big(
 \frac{2}{|B_h(t)|} \int_{-h}^h \Big|\frac{\sin (\gamma_{h,k}(t,T))}{\gamma_{h,k}(t,T)}
 \Big| \alpha_h(t,t+T)e^{-\mu(t+T)} dT\Big)
\end{equation}
where $\gamma_{h,k}(t,T)=2\pi k\alpha_h(t,t+T)/\ell$.
Now, if $|T|=|t-t'|\leq h/2$, then $\alpha_h(t,t')\geq \eps e^{|t|}h$ for some $\eps>0$ uniform in $t,t'$, thus
 $\gamma_{h,k}(t,T)\geq \eps e^{|t|}h$ for some $\eps>0$ uniform in $t$ and $k$, 
but since $|\sin(x)/x|\leq 1- \eps \min(x^2,1)$ if $\eps$ is chosen small enough above, 
one deduces that  
\[\sup_{|T|\leq h/2}\,  \sup_{t}\Big|\frac{\sin ( \gamma_{h,k}(t,T) )}{\gamma_{h,k}(t,T)}\Big|\leq 1-\eps h^2.\] 
Therefore, combining with \eqref{normLinfty}, we have that $||K_{h,k}f||_{L^\infty(\rr)}\leq A||f||_{L^\infty(\rr)}$  where 
\[A:= \sup_{t}\Big(\frac{ 2 }{|B_h(t)|}\int (\indic_{[0,h/2]}(|T|)(1-\eps h^2) + 
\indic_{[h/2,h]}(|T|))\alpha_h(t,t+T)e^{-\mu(t+T)} dT \Big)\]
and using \eqref{formulaball}, the integral $A$ can be bounded above as follows 
\[A\leq 1-\eps h^2 \frac{1}{|B_h(t)|}\int_{-h/2}^{h/2}2\alpha_h(t,t+T)e^{-\mu(t+T)} dT. \]
But now the integral on the right is exactly the volume for $dv_{g}$ of any region 
\[R(t,y_0):=\{(t',y'); |t-t'|\leq h/2, |y'-y_0|\leq \alpha_h(t,t') \}\]
when $y_0\in\rr/\ell\zz$.  When $t\leq \log(\ell/2\sinh(h))=:t_h$, we see directly that this region contains a geodesic ball of radius $\eps h$ centered at $(t,y_0)$ for some $y_0\in \rr/\ell\zz$ if $\eps$ is chosen small enough (note that $\eps=1/2$ 
works out when $t_h\geq |t|\geq t_0+1$), 
thus the volume is bounded below by $|B_{\eps h}(t)|$; when $|t|\geq t_h$, 
the region $R(t,y_0)$ contains a rectangle $\{|t-t'|\leq h/2, |y-y_0|\leq \alpha\}$ for some $\alpha>0$ independent of $h$, 
thus with volume $2\alpha\sinh(h/2)e^{-t}$, therefore $R(t,y_0)$ has volume bounded below by $|B_h(t)|/C$ for some $C>0$.  
Since we also have $|B_{\eps h}(t)|/|B_h(t)|\geq 1/C$ for some $C>0$ when $|t|\leq t_h$, we deduce that 
\[A\leq 1-\eps h^2/C.\]
which proves the first estimate of the Lemma. 
The $L^2(\rr,|B_h(t)|e^{-\mu(t)}dt)$ estimate \eqref{contraction2} can be obtained by interpolation. Indeed, since $K_{h,k}$ is self-adjoint 
with respect to the measure $|B_h(t)|e^{-\mu(t)}dt$ on $\rr$, 
the $L^\infty\to L^\infty$ operator bound implies that $K_{h,k}$ is bounded on $L^1(\rr, |B_h(t)|e^{-\mu(t)}dt)$ with norm bounded by
$A$, and by interpolation it is bounded on $L^2(\rr,|B_h(t)|e^{-\mu(t)}dt)$ with norm bounded by $A$.
 
Now for \eqref{contraction3}, we apply the same reasoning, but when $f$ is supported in $|t|\geq \tau$, we replace 
\eqref{normLinfty} by 
\[||K_{h,k}f||_{L^\infty}\leq ||f||_{L^\infty} \sup_{|t|\geq \tau-h}\Big(
 \frac{2}{|B_h(t)|} \int_{-h}^h \Big|\frac{\sin (\gamma_{h,k}(t,T))}{\gamma_{h,k}(t,T)}
 \Big| \alpha_h(t,t+T)e^{-\mu(t+T)} dT\Big)\]
and we use the same techniques as above except that now we use the bound 
\[\sup_{|T|\leq h/2}\,  \sup_{|t|\geq \tau}
\Big|\frac{\sin ( \gamma_{h,k}(t,T) )}{\gamma_{h,k}(t,T)}\Big|\leq 1- \eps \min(h^2 e^{2\tau}k^2, 1).\]
This yields an estimate 
\[ ||\indic_{|t|\geq \tau} K_{h,k}\indic_{|t|\geq \tau}||_{L^\infty\to L^\infty}\leq 1- \eps \min(h^2 e^{2\tau}k^2, 1)\]
and using self-adjointness of this operator and interpolation as above, we obtain the desired $L^2\to L^2$ estimate
for $\indic_{|t|\geq \tau} K_{h,k}\indic_{|t|\geq \tau}$. But this concludes the proof since this implies the same estimate 
(by changing $\eps$) on $K_{h,k}\indic_{|t|\geq \tau}=\indic_{|t|\geq \tau-h}K_{h,k}\indic_{|t|\geq \tau}$ if we take $\tau-h$
instead of $\tau$ above. 
\end{proof}

In the remaining part of the proof, we shall analyze the operator $K_{h,0}$ acting on functions 
constant in $y$. We split the surface in $3$ regions (see Figure \ref{fig:cusp-double}):
\[W_1:= \{ (t,y)\in (-t_h,t_h)\x \rr/\ell\zz \} \textrm{ with }\, t_h =\log(\ell/2\sinh(h))-1\]
 \[W_2:= \{ (t,y)\in (t_h,\infty)\x \rr/\ell\zz \}, \textrm{ and }  W_3:=\{ (t,y)\in (-\infty,-t_h)\x \rr/\ell\zz \}.\]
 

Let us define the functionals for $i=1,2,3$ acting on functions $f\in L^2(W,d\nu_h)$ which are constant in the $y$ variable 
\[\mc{E}_h^i(f):= \frac{1}{2Z_h}\int_{m,m'\in W_i, d(m,m')<h}(f(m)-f(m'))^2dv_g(m)dv_g(m')\]
\[V_h^i(f):=\frac{1}{2}\int_{m,m'\in W_i}(f(m)-f(m'))^2d\nu_h(m)d\nu_h(m').\]
Using the arguments used to obtain \eqref{Vhsplit} and \eqref{Ehcontrol}, we easily deduce that it suffices to prove that 
\[\mc{E}_h^1(f)\geq Ch^2 V^1_h(f) , \textrm{ and } \mc{E}^i_h(f)\geq Ch^2e^{t_h}V_h^i(f) \,\,\textrm{ for }i=2,3\]
hold for any $f\in L^2(W,d\nu_h)$ constant in the $\rr/\ell\zz$ variable to obtain, combined with \eqref{contraction2}, the 
estimate \eqref{EhVh}.\\

We start by the regions $W_2,W_3$, which are non-compact. We will reduce to a
random walk operator on the line with a measure decaying exponentially fast as $|t|\to \infty$.

\begin{lemma}\label{secondezone}
There exists $C>0$ such that for any $f\in L^2(W_2,d\nu_h)$ constant in the $\rr/\ell\zz$ variable  
\[\mc{E}_h^i(f)\geq Ch^2 e^{t_h}V_h^i(f), \,\textrm{ for }i=2,3.\] 
\end{lemma}
\begin{proof}
It suffices to prove the estimate for $i=2$, since clearly $i=3$ is similar.
Let $f$ be a function depending only on the variable $t$  and supported in $W_2$.
We first reduce the problem by changing variable:  we define $\til{f}(t):=f(t+t_h)$ on 
$\rr$ and using that 
$d\nu_h(t)/dtdy\leq Ce^{-2t}/h$ in $\{t\geq t_h\}$ and $e^{-t_h}=O(h)$ , we obtain 
\[ e^{t_h}V_h^2(f)\leq  Ce^{-t_h}\int_{t\geq 0,t'\geq 0} (\til{f}(t)-\til{f}(t'))^2 e^{-2(t+t')}dtdt'=:Ce^{-t_h}\til{V}_h^2(\til{f}).\]
Similarly, changing variable as above in $\mc{E}_h^2(f)$ and 
using the inclusion 
\begin{equation}\label{eq:incM2}
\begin{gathered}
\{(m,m')\in M_2\x M_2; |t(m)-t(m')|\leq h/2, |y(m)-y(m')|\leq \alpha\}\\
\subset \{(m,m')\in M_2\x M_2; d(m,m')\leq h\}
\end{gathered}
\end{equation}
for some $\alpha>0$ independent of $h$, we get 
\[\mc{E}^h_2(f)\geq \frac{e^{-2t_h}}{Z_h}\int_{t\geq 0,t'\geq 0,|t-t'|\leq h/2}
(\til{f}(t)-\til{f}(t'))^2 e^{-t-t'}dtdt'=:e^{-t_h}\til{\mc{E}}^{h/2}_2(\til{f})\]
We are thus reduced to prove an estimate of the form
\begin{equation}\label{toestimate}
\til{\mc{E}}^h_2(\til{f})\geq Ch^2 \til{V}^h_2(\til{f})
\end{equation}
for all $\til{f}\in C_0^\infty(\rr^+)$. Let $\rho=\rho(t)dt$ be a smooth non vanishing measure on $\rr$
equal to $e^{-t}dt$ on $(-1,\infty)$ and $e^{-\vert t\vert}$ on $(-\infty,-2)$ 
 and $d\nu^\rho_h(t):=\rho([t-h,t+h])\rho /Z^\rho_h$ where $Z_h^\rho$ is chosen such that $1=\int_\rr 1d\nu^\rho_h(t)$. 
In particular, $d\nu^\rho_h(t)=2e^{-2t}\sinh(h)dt/Z^\rho_h$ when $t\geq 0$ 
and $c_1h<Z^\rho_h<c_2h$ for some $c_1,c_2>0$.
Let us now define the self-adjoint one dimensional random walk 
operator $K^\rho_h$ on $L^2(\rr,d\nu^\rho_h)$ 
\[K^\rho_hf(t):=\frac{1}{\rho([t-h,t+h])}\int_{|t-t'|\leq h}f(t')\rho(t')dt'.\]
For $f$ supported in $\rr^+$, let $f^s$ 
be the even extension of $f$ to $\rr$.
Then  since $\rho$ doesn't vanish and is symmetric at infinity, there exists $C>0$ 
such that $\rho(t)/\rho(-t)\leq C$ and it is then easy to see (just like in the proof of Lemma \ref{t0-1}) 
that there exists $C>0$ such that
\[\begin{split}
\cjg (1-K^\rho_h)f^s,f^s\cjd_{L^2(\rr,d\nu^\rho_h)}=& \,
\frac{1}{Z^\rho_h}\int_{|t-t'|<h}(f^s(t)-f^s(t'))^2\rho(t)\rho(t')dtdt'\\
\leq & \frac{C}{Z^\rho_h}\int_{t\geq 0, t'\geq 0,|t-t'|<h}(f(t)-f(t'))^2e^{-(t+t')}dtdt'.
\end{split}\]
Since $e^{-t_h}=\beta \sinh(h)$ for some $\beta>0$, we deduce that there exists $C>0$ independent of $h$
such that for all functions $f$ compactly supported in $t>0$ and depending only on $t$
\[\til{\mc{E}}^{\frac{h}{2}}_2(f)\geq C\cjg (1-K^\rho_\frac{h}{2})f^s,f^s\cjd_{L^2(\rr,d\nu_h^\rho)}.\]
But we also notice that for the same class of functions
\[\til{V}^h_2(f)\leq C\int_{t,t'\in\rr}(f^s(t)-f^s(t'))^2d\nu^\rho_h(t)d\nu^\rho_h(t')
=C(\|f^s\|_{L^2(\rr,d\nu^\rho_h)}^2-\<f^s,1\>_{L^2(\rr,d\nu^\rho_h)}^2)
\] 
for some $C$, 
thus, to prove \eqref{toestimate}, it remains to show that 
\[\cjg(1-K^\rho_\frac{h}{2})f,f\cjd_{L^2(\rr,d\nu^\rho_h)} \geq Ch^2 (\|f\|_{L^2(\rr, d\nu^\rho_h)}^2-\<f,1\>_{L^2(\rr, d\nu^\rho_h)}^2).\]
We conclude by observing the measure $\rho(t)$ is tempered in the sense of \cite{GuMi}, 
hence the above estimate follows from Theorem 1.2 in \cite{GuMi} and the fact that $c_1<\frac{d\nu_{h/2}}{d\nu_h}<c_2$ for some $c_1,c_2>0$..
\end{proof}

And finally, we need to prove the last estimate:
\begin{lemma}\label{premierezone}
There exists $C>0$ such that for any $f\in C_0^\infty(W_1)$ depending only on $t$ 
\[\mc{E}_h^1(f)\geq Ch^2 V_h^1(f).\] 
\end{lemma}
\begin{proof}
We proceed in a way similar to the previous Lemma. We easily  notice from \eqref{ballh} the inclusion 
\[\begin{gathered} 
\{ (m;m')\in W_1\x W_1; |t(m)-t(m')|\leq h/2 , |y(m)-y(m')|\leq \alpha e^{|t|}h\}\\
\subset \{(m,m')\in W_1\x W_1; d(m,m')\leq h\}\end{gathered}\] 
for some $0<\alpha<1$ independent of $h$ and $t$, where $|y-y'|$ denotes the distance in $\rr/\ell\zz$. 
Consequently, since $dv_g(m)/dtdy \geq Ce^{-|t|}$, we have for any $f\in C_0^\infty(W_1)$ depending only on $t$
\begin{equation}
\begin{split}\label{eh1}
\mc{E}_h^1(f)= &\frac{1}{2Z_h}\int_{t(m),t(m')\in [-t_h,t_h], d(m,m')\leq h}(f(m)-f(m'))^2 dv_g(m)dv_g(m')\\
\geq & \frac{C}{Z_h}\int_{t,t'\in [-t_h,t_h], |t-t'|\leq h/2}(f(t)-f(t'))^2 e^{-|t|-|t'|}\alpha he^{|t|}dtdt'\\
\mc{E}_h^1(f)\geq & \frac{C}{h}\int_{t,t'\in [-t_h,t_h], |t-t'|\leq h/2}(f(t)-f(t'))^2 e^{-\frac{|t|}{2}-\frac{|t'|}{2}}dtdt'.
\end{split}
\end{equation}
Let $\rho:=\rho(t)dt$ be a smooth positive measure on $\rr$ defined like in the proof of Lemma \ref{secondezone}
 but with $\rho(t)=e^{-|t|/2}$ in $\rr\setminus (-1,0)$ instead of $e^{-|t|}$. 
Let us define the random walk operator on $\rr$ 
\[K^\rho_h(f)(t)=\frac{1}{\rho([t-h,t+h])}\int_{|t-t'|< h}f(t')\rho(t')dt'\] 
which is self-adjoint on $L^2(\rr, d\nu_h^\rho(t))$ if $d\nu^\rho_h(t):=\frac{\rho([t-h,t+h])}{Z^\rho_h}\rho(t)dt$ 
and $Z^\rho_h$ is chosen such that $d\nu^\rho_h$ is a probability measure (in particular $c_1h<Z^\rho_h<c_2h$).
For $f$ supported in $[-t_h,t_h]$, let $f^p$ be the periodic 
extension of $f$ defined by $f^p(2j t_h+t):=f(t)$ when $t\in[-t_h,t_h]$ and $j\in\zz$.
We set for $g\in L^2(\rr)$
\[\mc{E}^\rho_h(g):=\cjg (1-K^\rho_h)g,g\cjd_{L^2(\rr,d\nu^\rho_h)}=\frac{1}{Z^\rho_h}
\int_{t,t'\in \rr, |t-t'|< h}(g(t)-g(t'))^2 \rho(t)\rho(t')dt dt'.\]
For $j\in\nn$, let $F_j=2jt_h+[-t_h,t_h]$. Using the changes of variable $t\mapsto t+2jt_h$ and  $t'\mapsto t'+2kt_h$, 
we get 
\[\begin{split}
\mc{E}^\rho_h(f^p)&\leq \frac{C}{Z^\rho_h}\sum_{k,j=0}^\infty\int_{t\in F_k,t'\in F_j, |t-t'|< h}(f^p(t)-f^p(t'))^2 e^{-\frac{|t|}{2}-\frac{|t'|}{2}}dtdt'\\
&\leq \frac{C}{Z^\rho_h}\sum_{k,j=0}^\infty \int_{t\in F_0,t'\in F_0, |t-t'|< h}
(f(t)-f(t'))^2 e^{-\frac{|t+2jt_h|}{2}-\frac{|t'+2kt_h|}{2}}dtdt'\\
&\leq \frac{C}{Z^\rho_h}\int_{t,t'\in [-t_h,t_h], |t-t'|< h}
(f(t)-f(t'))^2 e^{-\frac{|t|}{2}-\frac{|t'|}{2}}dtdt'
\end{split}\]
where we have use in the last line that for $t\in F_0$ and any $j\in\zz$
\[e^{-\frac{|t+2jt_h|}{2}}=\left\{\begin{array}{ll}
e^{-t/2-jt_h} & \textrm{ if }j>0\\
e^{t/2+jt_h} & \textrm{ if }j<0\\
e^{-|t|/2} & \textrm{ if }j=0
\end{array} \right. \leq  e^{-|t|/2}e^{-(|j|-1)t_h}.\]
Since $c_2h\geq Z^\rho_h\geq c_1 h$, this shows using \eqref{eh1} that $\mc{E}_{\frac h 2}^\rho(f^p)
\leq C \mc{E}_h^1(f)$.
Moreover, defining $V_h^\rho(f)=\|f\|^2_{L^2(d\nu^\rho_h)}-\<f,1\>_{L^2(\rr,d\nu^\rho_h)}$, 
and using that for $|t|\leq t_h$, $ \rho([t-h/2,t+h/2])\geq C\sinh(\frac h 2)e^{-|t|/2}$ for some $C$, 
we have 
\begin{align}
V^\rho_{\frac h
  2}(f^p)&=\int_{t,t'\in\rr}(f^p(t)-f^p(t'))^2d\nu^\rho_{\frac h
  2}(t)d\nu^\rho_{\frac h 2}(t') \label{E:Vrho} \\
&\geq C\int_{t,t'\in[-t_h,t_h] }(f(t)-f(t'))^2
e^{-|t|}e^{-|t'|}dtdt'. \notag
\end{align}
Since $V_h^1(f)$ is easily seen to be bounded above by $C$ times the
right hand side of \eqref{E:Vrho} (in view of the assumptions
on the metric $g$ on $W_1$), this shows that $V_{\frac h 2}^\rho (f^p)\geq CV_h^1(f)$. 
Combining this with the estimate on Dirichlet form, it remains to show 
that $V^\rho_{\frac h 2}(f)\leq Ch^2 \mc{E}^\rho_{\frac h 2}(f)$. 
Since the measure $\rho$ is tempered in the sense of \cite{GuMi}, 
this is again a consequence of Theorem 1.2 of this paper.
\end{proof}
Combining Lemmas, \ref{secondezone}, \ref{premierezone} and \ref{contraction}, we have proved the estimate 
\eqref{EhVh} .
\end{proof}

\section{Upper bound on the gap and discrete eigenvalues of the Laplacian}

In this section, we shall give a sharper upper bound on the gap $g(h)$ when the Laplacian 
has an eigenvalue smaller than $4/3$ (beside $0$).
More precisely, we are going to prove the following
\begin{theorem}\label{spectre}
Let $0=\la_0< \la_1\leq \dots \leq \la_K$ be the $L^2$ eigenvalues of the Laplacian $\Delta_g$ on $(M,g)$
which are contained in $[0,1/4)$ and $\la_{K+1},\dots,\la_{K+L}$ those contained in 
$[1/4,4/3)$. Then for all $c>0$, there is $h_0$ such that for all
$h\in (0,h_0)$ and $K+1\leq k\leq k+L$
\[\sharp \Big({\rm Spec}(1-K_h)\cap \Big[\frac{\la_k h^2}{8}- ch^4,\frac{\la_k h^2}{8}+ch^4\Big]\Big)
\geq \dim {\rm ker}(\Delta_g-\la_k).\] 
For all $c>0$ there exists $h_0$ 
such that for all $0<h<h_0$ and $0<k\leq K$, 
\[\sharp \Big({\rm Spec}(1-K_h)\cap \Big[\frac{\la_kh^2}{8}- ch^{2+\sqrt{1/4-\la_k}},\frac{\la_kh^2}{8}+ch^{2+\sqrt{1/4-\la_k}}\Big]\Big)
\geq \dim {\rm ker}(\Delta_g-\la_k).\] 
\end{theorem}

We shall first need a few results relating $K_h$ to the Laplacian and some estimates on 
the eigenfunctions of $\Delta_g$ in the cusp.

\subsection{Asymptotic expansion in $h$ of $K_h\psi$}

\begin{lemma}\label{approxcompact}
For all $\tau>t_0$, there is $C>0$ and $h_0>0$ such that for any $\psi\in C_0^\infty(M)$ with support in $\{t< \tau\}$ for $h\in(0, h_0)$
\begin{equation}\label{khpsi}
\left|\left|K_h\psi -(\psi -\frac{h^2}{8}\Delta_g \psi)\right|\right|_{L^2(M)}\leq Ch^4||\psi||_{H^4(M)}.
\end{equation}
\end{lemma}
\begin{proof} If the cusp is denoted by $[0,\infty)_t\x \rr/\ell\zz$,  
the support of $\psi$ is contained in $\{t<\tau\}$ for some $\tau>0$. Let us define a smooth Riemannian compact surface 
$(X,g_X)$ which is obtained by cutting the cusp end $\{t>\tau+1\}$ of $M$ and gluing instead a half sphere, and such that the metric $g_X$ on $X$ is an extension of the metric $g$ in the sense that $g_X$ is isometric to $g$ in $t\leq \tau+1$. 
Then, since the support of $K_h\psi$ is larger than ${\rm supp}(\psi)$ by at most a set of 
diameter $h$, one has that  for $h\ll e^{-\tau}$, the function $K_h\psi$ has support inside $\{t\leq \tau+h\}$ and thus can be considered 
as a function on $X$ in a natural way, and it is given by $K^X_h\psi$ 
where $K^X_{h}$ is the random walk operator associated to $(X,g_X)$.  
We can use the results of Lebeau-Michel \cite{LebMi}, i.e. Lemma 2.4 of this article which describes $K^X_h$ as a semiclassical
pseudo-differential operator on $X$, in particular this provides the expansion of 
the operator $K^X_h$ in powers of $h$ to fourth order, and shows \eqref{khpsi} when acting on smooth functions $\psi$. 
\end{proof}

In the next lemma we give an approximation for functions
 supported in the region where the geodesic balls of radius $h$ do not overlap.
\begin{lemma}\label{belowcut}
Let us choose $t_0>0$ such that the metric $g$ is constant curvature in the region $\{t>t_0/2\}$ of the 
cusp and let $h\in (0,h_0)$ where $h_0$ is fixed small. Consider $\chi_h\in C_0^\infty(M)$ 
supported in $\{e^{t_0}\leq e^t\leq \frac{\ell}{2\sinh(h)}-1\}$, 
and $\chi_h$ depending only on $t$ with 
$||\pl_t^j\chi_h||_{L^\infty}\leq C_j$ for all $h\in (0,h_0)$ and all $j\in \nn_0$. Then there is $C>0$ such that for all $\psi \in C^\infty(M)$ 
and all $h\in(0,h_0)$
\[\left|\left|K_h(\psi\chi_h)-\Big(\psi\chi_h -\frac{h^2}{8}\Delta_g (\psi\chi_h)\Big)\right|\right|_{L^2(M)}\leq 
Ch^4||\psi||_{H^4(M_h)}\]
where $M_h:=\{e^t\leq \frac{\ell}{2|\sinh(h)|}-1\}$.
\end{lemma}
\begin{proof} Let us use the coordinates $(x=e^t,y)$ in the half-plane model of $\hh^2$ and define $x_0:=e^{t_0}$ and $x(h):=\frac{\ell}{2\sinh(h)}-1$. Let $\phi_r$ be smooth  and supported in the part $r/2\leq x\leq 2r$ of the cusp 
where $r\in (x_0,x(h))\cap \nn$ is fixed.  
Consider $\til{\phi}$ the lift to $\hh^2$, i.e. $\til{\phi}$ is periodic under the translation
$\gamma:y\to y+\ell$ and projects down to $\phi$ under the quotient of $\hh^2$ by this translation. 
If $\til{K}_h$ denotes the random walk operator on $\hh^2$, we have that $\til{K}_h\til{\phi}$ 
is periodic under $\gamma$ and $K_h\phi$ is its projection under the quotient map. 
The squared Sobolev norm  $||\phi||^2_{H^k(\mc{C})}$ (for $k\in\nn_0$) of a smooth function $\phi$ in the cusp 
$\mc{C}=\cjg \gamma\cjd\backslash\hh^2$ supported in $r/2<x<2r$ is equal to $\frac{1}{r}||\til{\phi}||^2_{H^k(W_r)}$ where
$W_r=\{(x,y)\in \hh^{2}; x\in(\demi r,2r), |y|\leq r\ell\}$. Let $G_r$ be the isometry $(x,y)\to \frac{1}{\ell r}(x,y)$ 
of $\hh^2$ which maps $W_r$ to a domain included in a geodesic ball $B_0$ of $\hh^2$ centered at $(1,0)$ and of radius independent of $r$ and $h$. 
Now it is clear that $G_r^*\til{K}_h{G_r^{-1}}^*=\til{K}_h$ since $G_r$ is an isometry of $\hh^2$. 
From Lemma 2.4 of \cite{LebMi}, which is purely local, we deduce that for $u\in C^\infty(\hh^2)$, we have
\[ ||\til{K}_hu-u-\frac{h^2}{8}\Delta_{\hh^2}u||_{L^2(B_0)}\leq Ch^4||u||_{H^4(B_1)}\]
where $B_1$ is a hyperbolic geodesic ball centered at $(1,0)$ containing $B_0$ and of 
Euclidean radius $\alpha$ for some $\alpha>0$ independent of $h,r$. 
Since $G_r^*$ commutes also with $\Delta_{\hh^2}$ and since it is also an isometry for the $L^2(\hh^2)$ and $H^4(\hh^2)$ norms, 
we deduce easily that  
\[||\til{K}_h\til{\phi}-\til{\phi}-\frac{h^2}{8}\Delta_{\hh^2}\til{\phi}||_{L^2(W_r)}\leq Ch^4||\til{\phi}||_{H^4(W_{\beta r})}\]
for some $\beta>0$ independent of $r,h$, which implies directly 
\[ ||K_h\phi-\phi-\frac{h^2}{8}\Delta_{g}\phi||_{L^2(\mc{C})}\leq Ch^4 \sqrt{\beta} ||\phi||_{H^4(\mc{C})}\]
and thus the desired result for a function supported in $\{r/2\leq x\leq 2r\}$ in the cusp. Now it suffices
to sum over a dyadic covering of the region $\{x_0\leq x\leq x(h)\}$ of the cusp.  
\end{proof}

We end this part with another estimate in the part of the cusp where the balls $B_h(t)$ 
overlap:
\begin{lemma}\label{estimabove}
Let $A\gg 0$, then there is $C>0$ and $h_0>0$ such that for all 
smooth function $\psi$ supported in  $\{\frac{A}{\sinh(h)}\geq e^t\geq \frac{\ell}{2\sinh(h)}-2\}$
depending only on the variable $t$ and all $h\in (0,h_0)$ 
\[||K_h\psi-\psi||_{L^2(M,dv_g)} \leq Ch^2||\psi||_{H^2(M,dv_g)} \]
\end{lemma}
\begin{proof} Using the fact that $\psi$ depends only on $t$, a Taylor expansion of $\psi$ gives  
$\psi(t+T)=\psi(t)+T\pl_t\psi(t)+T^2Q_T\psi(t)$ with $Q_T\psi(t)=\demi \int_{0}^1(1-u)^2\pl_t^2\psi(t+Tu)du$ for $T$ small, 
then we can use the expressions
\eqref{Kh1} and \eqref{Kh2} to deduce that    
\[K_h\psi(t)=\psi(t)+\alpha_h\pl_t\psi(t)+R_h(t)\]
with $\alpha_h$ given, for $e^t\sinh(h)\leq \ell/2$ by 
\[\alpha_h=\frac{1}{4\pi(\sinh(h/2))^2}\int_{-\sinh(h)}^{\sinh(h)}
\int_{\log(\cosh(h)-\sqrt{\sinh(h)^2-|z|^2})}^{\log(\cosh(h)+\sqrt{\sinh(h)^2-|z|^2})}
Te^{-T}dTdz\]
and for $e^t\sinh(h)\geq \ell/2$
\[\alpha_h=\frac{1}{|B_h(t)|}\int_{-\frac{e^{-t}\ell}{2}}^{\frac{e^{-t}\ell}{2}}
\int_{\log(\cosh(h)-\sqrt{\sinh(h)^2-|z|^2})}^{\log(\cosh(h)+\sqrt{\sinh(h)^2-|z|^2})}
Te^{-T}dTdz\]
while the $R_h(t)$ term satisfies the bound for $e^t\sinh(h)\leq \ell/2$ (here 
the Sobolev norms are taken with respect to the measure $e^{-t}dt$)
\[\begin{split}
||R_h||_{L^2}\leq &\frac{C||\psi||_{H^2}}{4\pi(\sinh(h/2))^2}\int_{-\sinh(h)}^{\sinh(h)}
\int_{-h}^{h}
T^2e^{-T}dTdz\\
 \leq & Ch^2||\psi||_{H^2}
\end{split}\]
and for $e^t\sinh(h)\geq \ell/2$
\[\begin{split}
||R_h||_{L^2}\leq &C\Big|\Big|\frac{\pl_t^2\psi(t)}{e^t|B_h(t)|}\Big|\Big|_{L^2(e^{-t}dt)}\int_{-\ell/2}^{\ell/2}
\int_{-h}^{h}
T^2e^{-T}dTdz\\
 \leq & Ch^2||\psi||_{H^2}
\end{split}\]
where we used that $|B_h(t)|\geq ce^{-t}h$ for some $c>0$ combined with   
 the fact that $T^2e^{-T}$ is increasing for $T<2$. 
Now we have to evaluate $\alpha_h$. Let us write the part $e^t\sinh(h)\geq \ell/2$, the other one 
being even simpler, and this can be done by observing that a primitive of $Te^{-T}$ is given by $-(1+T)e^{-T}$
\[|\alpha_h|\leq c\frac{e^{t}}{h}\int_{-\frac{e^{-t}\ell}{2}}^{\frac{e^{-t}\ell}{2}}
|(1+t_+(z))e^{-t_+(z)}-(1+t_-(z))e^{-t_-(z)}|dz \] 
where $t_\pm(z)=\log(\cosh(h)\pm \sqrt{\sinh(h)^2-|z|^2})$. We can remark that 
\[t_\pm(z)=\pm \sqrt{\sinh(h)^2-|z|^2} +O(h^2)\]
uniformly in $|z|\leq \sinh(h)$ and thus 
\[|(1+t_+(z))e^{-t_+(z)}-(1+t_-(z))e^{-t_-(z)}|=|t_+(z)^2-t_-(z)^2|+O(h^3)=O(h^3),\]
proving that $|\alpha_h|=O(h^2)$. This achieves the proof.
\end{proof}

\subsection{The Laplacian eigenfunctions}
For a surface with hyperbolic cusps, 
the spectral theory of the Laplacian $\Delta_g$ is well known (see for instance \cite{Mu}).
The essential spectrum of $\Delta_g$ is given by $\sigma_{\rm ess}(\Delta_g)=[1/4,\infty)$,
there are finitely many $L^2$-eigenvalues $\la_0=0,\la_1,\dots, \lambda_K$ 
in $[0,1/4)$ and possibly infinitely many embedded eigenvalues $(\lambda_{j})_{j\geq K+1}$ in $[1/4,\infty)$.
Moreover one has
\begin{lemma}\label{esteigenlap}
Let $T\gg 0$ be large and $\chi_T$ be a smooth function supported in $\{t\geq T\}$. 
The $L^2(M,dv_g)$ normalized eigenfunctions associated to $\lambda_j$ with $j>K$ satisfy the estimates in the cusp 
\begin{equation}\label{embedded}
||\chi_T\psi_j||_{L^2(M,dv_g)}\leq C_{N,j}e^{-NT}, \quad \forall N\in \nn_0, \forall T\gg 0
\end{equation}
for some constants $C_{N,j}$ depending on $N,j$. 
The normalized eigenfunctions $\psi_j$ for an eigenvalue $\la_j\in [0,1/4)$ satisfy for some $C_j>0$ depending on $j$
\begin{equation}\label{notembedded}
||\chi_T\psi_j||_{L^2(M,dv_g)}\leq C_je^{-T\sqrt{1/4-\la_j}}, \quad 
\forall T\gg 0.
\end{equation}
\end{lemma}
\begin{proof} This is a well known fact, but we recall the arguments for the convenience of the reader. 
We use the Fourier decomposition in the $\rr/\ell\zz$ variable of the cusp $\mc{C}:=[t_0,\infty)_t\x (\rr/\ell\zz)_\theta$ and, since the metric 
is isometric to $dt^2+e^{-2t}d\theta^2$, the operator $\Delta_g$ decomposes as the direct sum of operators 
\[\begin{gathered}
e^{-\frac{t}{2}}\Delta_g\Big(e^{\frac{t}{2}}\sum_{k\in\zz} u_k(t)e^{\frac{2i\pi ky}{\ell}}\Big)=
\sum_{k\in\zz}P_ku_k(t)e^{\frac{2i\pi ky}{\ell}},\\
P_ku(t)=\Big(-\pl_t^2+\frac{4\pi^2k^2}{\ell^2} e^{2t}+\frac{1}{4}\Big)u(t).
\end{gathered}\]
and the $L^2(\mc{C})$ space in the cusp decomposes as $L^2(\mc{C})\simeq \oplus_{k\in\zz}H_k$ where 
$H_k\simeq L^2([t_0,\infty),dt)$. 
We decompose a normalized eigenfunction $\psi_j$ for the eigenvalue $\la_j$ into the form $u_0(t)+\phi_j(t,y)$
where $u_0$ is the $k=0$ component of $\psi_j$ in the Fourier decomposition. 
When $u$ is a function supported in the cusp and with only $k\not=0$ components, 
we observe that $\cjg P_ku,u\cjd\geq  Ce^{2T}||u||^2_{L^2}$ and so 
if $\chi_T$ is a function which is supported in $\{t\geq T\}$ we use the fact that 
$||\phi_j||_{H^n(M)}\leq C(1+\la_j)^n$ for all $n\in\nn_0$, we deduce that for all $N\in\nn_0$
\[ ||\chi_T\phi_j||_{L^2}\leq C_{N,j}e^{-NT}\]
for some constants $C_{N,j}$ depending on $N,j$.
Now the $k=0$ component are solutions of $(-\pl_t^2-\la_j+1/4)u(t)=0$, 
and there is a non-zero $L^2$ solutions in the cusp only if $\la_j\in[0,1/4)$, and they are given by
\[u(t)=B e^{-t\sqrt{1/4-\la_j} }, \quad B\in\cc\]
this achieves the proof.
\end{proof}

\subsection{Proof of Theorem \ref{spectre}}
We are now in position to prove the Theorem. Let $\psi_k$ be an $L^2$ eigenfunction for 
$\Delta_g$ with eigenvalue $4/3>\la_k>1/4$. By Lemma \ref{esteigenlap} with 
$T=|\log h|/4$ and $N>16$ we see that $||K_h\chi_T\psi_k||_{L^2}=O(h^4)$ where $\chi_T$ is a cutoff which is equal to $1$
in $\{t\geq T+1\}$. With $t_0>0$ chosen like in Lemma \ref{belowcut}, we let $\chi_0+\chi_1+\chi_T=1$ be a partition of unity associated to $\{t\leq t_0\}\cup \{T\geq t\geq t_0\}\cup \{t\geq T\}$ and let $\til{\chi}_j$ equal to $1$ on the a region containing 
$\{m\in M; d(m,{\rm supp}\chi_j)\leq 1\}$  and with support in  $\{m\in M; d(m,{\rm supp}\chi_j)\leq 2\}$ (for $j=0,1,T$).
Since $K_h$ propagates the support at distance $h<1$ at most,   
we can write 
\[  (K_h-1+h^2\la_k/8)\psi_k=\sum_{j=0,1,T}
\chi_j (K_h-1+h^2\Delta_g/8)\til{\chi}_j\psi_k.\]
We can then combine this with the result of Lemma \ref{belowcut} and \ref{approxcompact} (since $||\psi_k||_{H^4}\leq C\la^2_k$) and  Lemma \ref{esteigenlap} to obtain by partition of unity
 \[||K_h\psi_k-(1-h^2\frac{\la_k}{8})\psi_k||_{L^2}\leq Ch^4.\]
By applying the spectral theorem above the essential spectrum of $K_h$, 
this implies that for all $c>0$, there is $h_0$ such that for all $h\in (0,h_0)$
with $1-h^2(\la_k/8+ch^2)> h/\sinh(h)$,
\[\sharp \Big({\rm Spec}(h^{-2}(1-K_h))\cap \Big[\frac{\la_k}{8}- ch^2,\frac{\la_k}{8}+ch^2\Big]\Big)
\geq \dim \ker(\Delta_g-\la_k).\]

It remains to deal with the orthonormalized eigenfunctions $\psi_j$ of $\Delta_g$ for eigenvalues $\la_j\in [0,1/4)$.
We proceed as before but we use a partition of unity $\sum_{j=0}^3\chi_j=1$
associated to 
\[\{t\leq t_0\}\cup \{t_0\leq t\leq t_1=\log(2/\ell\sinh(h))-1\}\cup \{ t_1\leq t\leq t_2=A\log(1/h)\}\cup
\{t\geq t_2\}\]
for some large $A>0$ independent of $h$.
By Lemmas \ref{belowcut}, \ref{approxcompact} and the arguments used above, we have  
\[||(\chi_0+\chi_1)(K_h-1+h^2\frac{\la_k}{8})\psi_k||_{L^2}\leq Ch^4,\]
then by Lemma \ref{estimabove} one has for $\til{\chi}_2$ defined like above (but for $\chi_2$)
\[||\chi_2 (K_h-1+h^2\frac{\la_k}{8})\psi_k||_{L^2}\leq Ch^2||\til{\chi}_2\psi_k||_{H^2}=O(h^{2+\sqrt{1/4-\la_k}})\]
where we have use \eqref{notembedded} for the last estimate, and we finally have for 
$\til{\chi}_3$ defined like above but with respect to $\chi_3$
\[||\chi_3 (K_h-1+h^2\frac{\la_k}{8})\psi_k||_{L^2}\leq C||\til{\chi}_3\psi_k||_{L^2}=O(h^{A\sqrt{1/4-\la_k}})\]
as a consequence of  \eqref{notembedded}. Taking $A\sqrt{1/4-\la_k}\geq 3$, 
this achieves the proof of Theorem \ref{spectre} by the same arguments as above.

\section{Total Variation estimates}
In this section we address the problem of getting some estimate on the difference beetwen the iterated Markov kernel and its stationnary measure, in the total variation norm.
Recall that since $K_h$ is selfadjoint on $L^2(M,d\mu_h)$ and $K_h(1)=1$, then $d\nu_h$ is a stationary measure for $K_h$.
Let us recall that if $\mu$ and $\nu$ are two probability measure on a set $E$, their total variation distance is defined by 
\begin{equation*}
\|\mu-\nu\|_{TV}=\sup_{A}|\mu(A)-\nu(A)|
\end{equation*}
where the sup is taken over all measurable sets. Then, a standard compuation shows that
\begin{equation}\label{eq:equivTV}
\|\mu-\nu\|_{TV}=\frac 1 2\sup_{\parallel f\parallel_{L^\infty=1}}\vert\mu(f)-\nu(f)\vert
\end{equation}

Until the end of this section, we use the function $m\in M\mapsto t(m)\in [0,\infty[$ defined in the proof of Theorem \ref{gaph^2}. For $\tau\geq t_0$, let 
$M_\tau=\{m\in M,\,t(m)\leq \tau\}$.

\begin{theorem}\label{th:TV_estimates}
There exists $h_0>0$ such that the following hold true:
\begin{enumerate}
\item[i)] There exists $C>0$ such that for all $h\in]0,h_0]$ and $n\in\nn$,
\be
\sup_{m\in M_\tau}\|K_h^n(m,dm')-d\nu_h\|_{TV}\leq C \max(h^{-1},h^{-\frac 1 2}e^{\frac\tau 2})e^{-ng(h)}
\ee

\item[ii)] There exists $C>0$ such that for any $h\in]0,h_0]$ and $n\in\nn$, there exists $m\in M_{2nh}$ such that
\be
\|K_h^n(m,dm')-d\nu_h\|_{TV}\geq 1-Ch^{-1}e^{-2nh}
\ee
\end{enumerate}

\end{theorem}

\begin{proof} Let $h_0>0$ such that the results of the previous
  sections hold true, and define the orthogonal projection $\Pi_0$
  onto the subspace of constant functions in $L^2( d \nu_h)$:
\[
\Pi_0 (f) = \int_M f(m) d \nu_h(m).
\]
Let us start with the proof of i). Let $\tau\geq t_0$ be fixed.  Thanks to \eqref{eq:equivTV}, we have
\begin{equation}
\begin{split}
\sup_{m\in M_\tau}\|K_h^n(m,dm')-d\nu_h\|_{TV}&=\frac 1 2 \sup_{m\in M_\tau}\sup_{\|f\|_{L^\infty(M)}=1}|K_h^n(f)(m)-\Pi_0(f)|\\
&=\frac 1 2\|\indic_{M_\tau}(K_h^n-\Pi_0)\|_{L^\infty(M)\rightarrow L^\infty(M)}
\end{split}
\end{equation}
Denote  $E_\lambda$ the spectral resolution of $K_h$. From the spectral theorem combined with
Theorem \ref{gaph^2} and Proposition \ref{bottom} we have
  \begin{equation*}
  K_h^{n-2}-\Pi_0=\int_{-1+\delta}^{1-g(h)}\lambda^{n-2}dE_\lambda,
   \end{equation*}
and hence  $\|K_h^{n-2}-\Pi_0\|_{L^2(d\nu_h)\rightarrow L^2(d\nu_h)}\leq e^{-ng(h)}$. Moreover, 
$\|K_h-\Pi_0\|_{L^\infty\rightarrow L^2}\leq 2$ and we have only to show that
$\|\indic_{M_\tau}(K_h-\Pi_0)\|_{L^2\rightarrow L^\infty}\leq C h^{-\frac 1 2}e^{\frac\tau 2}$.
For this purpose, let $f\in L^2(M,d\nu_h)$ be such that $\|f\|_{L^2}=1$. Then
\[
|\Pi_0(f)|\leq \|f\|_{L^2(d\nu_h)}\nu_h(M)^\demi=1
\]
and it remains to estimate $\indic_{M_\tau}K_h f$. For $m\in M_\tau$, we have
\[
K_hf(m)=\frac 1 {|B_h(m)|}\int_{B_h(m)}f(y)dv_g(y)=\frac 1 {|B_h(m)|}\int_{B_h(m)}f(m')\frac {Z_h}{|B_h(m')|}d\nu_h(m')
\]
hence,
\[
|K_hf(m)|\leq \|f\|_{L^2(d\nu_h)}\frac 1 {|B_h(m)|}\Big(\int_{B_h(m)}\frac {Z_h^2}{|B_h(m')|^2}d\nu_h(m')\Big)^{\frac 1 2}
\]
If $t(m)\leq \log(\ell/2\sinh(h))$, since $|B_h(m)|\geq Ch^2$, we get $|K_hf(m)|\leq Ch^{-1}$.

If $t(m)\geq \log(\ell/2\sinh(h))$, since $|B_h(m)|\geq Che^{-t(m)}$ and $d\nu_h(t,y)\leq Che^{-2t}dtdy$, an easy calculation shows that
$|K_hf(m)|\leq Ch^{-\frac 1 2}e^{\frac\tau 2}$ and the proof of i) is complete.\\

Let us prove ii). Let $n\in \nn$ and $m_{n,h}\in M$ such that $t(m)=2nh$. Let $f_{n,h}(m)=\indic_{t(m)>nh}-\indic_{t(m)<nh}$. Then $\parallel f _{n,h}\parallel_{L^\infty}=1$ and $K_h^nf_{n,h}(m_{n,h})=1$.
On the other hand, $\Pi_0(f _{n,h})=-1+2\int_{t(m)\geq nh}d\nu_h(m)=-1+O(h^{-1}e^{-2nh})$.  Therefore, 
$K_h^n(f_{h,n})(m_{h,n})-\Pi_0(f _{n,h})=2+O(h^{-1}e^{-2nh})$ and the proof is complete.
\end{proof}

\section{Smoothing estimates for $K_h$}
In this last section, we shall show that $K_h$ regularizes $L^2$ functions in the sense that 
it gains $1$-derivative. In particular this implies that the eigenfunctions of $K_h$ are in $H^1(M)$.
It is actually possible to prove $C^\infty$ regularity of eigenfunctions outside the 
line $t=\log(\ell/2\sinh(h))$ where the balls start to overlap, but we
do not include it here since it is quite technical
and not really useful for our purpose. On the other hand, it is unlikely to get much better than $H^1$ or $H^2$ global regularity 
for eigenfunctions since the operator itself (as a Fourier integral operator) has a singularity at $t=\log(\ell/2\sinh(h))$, as well as 
the volume of the ball $|B_h(m)|$.
\begin{proposition}\label{smoothing} 
There exists $C>0$ and $h_0>0$ such that for all $0<h<h_0$ and $f\in L^2(M,dv_g)$
\begin{equation}\label{smoothingH1}
||K_h f||_{H^1(M,dv_g)}\leq Ch^{-1}||f||_{L^2(M,dv_g)}
\end{equation}
 where the Sobolev norm $H^1$ is taken with respect to the metric $g$. 
\end{proposition}
\begin{proof}
If $M_0=\{m\in M ; t(m)\leq t_0\}$ is a compact part such that $M\setminus M_0$ is isometric to the
cusp $(t_0,\infty)_t\x (\rr/\ell\zz)_y$ with metric $dt^2+e^{-2t}dy^2$ as before, 
then the estimate \eqref{smoothingH1} for $f$ supported in $M_0$ (or a slightly bigger compact set in general) 
is proved in \cite{LebMi} using microlocal analysis. It then remains to analyse the cusp part.
We decompose the proof in two Lemmas. 
\begin{lemma}
Let $L\geq \ell/2$ and $t_0>0$ as above 
then for any $f\in L^2$ supported in the region $\{ t_0\leq t\leq \log(L/\sinh(h))\}$, we have
\[||\pl_tK_hf||_{L^2(M,dv_g)}\leq Ch^{-1}||f||_{L^2(M,dv_g)}\]
while for all $f\in L^2$ supported in $\{t\geq t_0\}$ 
\[||e^t\pl_yK_hf||_{L^2(M,dv_g)}\leq Ch^{-1}||f||_{L^2(M,dv_g)}.\]
\end{lemma}
\begin{proof}
We shall use the Fourier decomposition in the $\rr/\ell\zz$ variable and the expression of $K_h$
in Subsection \ref{secondexp} according to this decomposition.
Let us start with the part $e^t\pl_yK_h$. Since $e^t\pl_y$ amounts to multiplication by $2\pi ike^t/\ell$ on the Fourier 
$k$-th mode in $y$, it suffices to get a bound of the form 
\[||e^{t}kK_{h,k}f_k(t)||_{L^2(e^{-t}dt)}\leq Ch^{-1}||f_k(t)||_{L^2(e^{-t}dt)},\]
but this is straightforward from the expression \eqref{knot=0} by using $||f(\cdot+T)||_{L^2(e^{-t}dt)}=||f||_{L^2(e^{-t}dt)}e^{T/2}$, 
the fact that the size of integration in $T$ is less than $h$ and
$|B_h(t)| \geq \eps e^{-t}h$ for some $\eps>0$ in the region $\{e^t\sinh(h)\geq \ell/2\}$. 
Now we have to consider the operators with $\pl_tK_{h,k}$, say acting on smooth functions, 
and this needs a bit more care because of the lack of smoothness
on the line $\{e^t\sinh(h)=\ell/2\}$. First, observe that $|B_h(t)|$ is a $C^1$ function of $t$, which is smooth outside
$\{e^t\sinh(h)=\ell/2\}$, and we have $\pl_t|B_h|/|B_h|\in [0,\eps^{-1}]$ for some $\eps>0$, 
this follows directly from the explicit formula \eqref{eq:formule_volume}. As a consequence, when the derivative $\pl_t$
hits $|B_h(t)|^{-1}$ or $e^{-t-T}$ in \eqref{knot=0} or in \eqref{k=0}, 
one obtains terms which are estimated like we did above for $ke^tK_{h,k}$.
Now let us assume $e^t\sinh(h)\geq \ell/2$. Then using $\alpha(\log T_\pm(t))=e^{-t}$ we have 
$\sin(\pi ke^t\alpha(\log T_\pm(t)))=0$ and we thus obtain from \eqref{knot=0} that for $k\not=0$ 
\begin{equation}\label{plt}
\begin{gathered}
\frac{\pl_t(|B_h|e^{t}K_{h,k}f(t))}{|B_h|e^t}=(K_{h,k}\pl_tf)(t)\\
+\ell |B_h|^{-1}\int_{-h}^{\log T_-(t)}+\int_{\log T_+(t)}^h f(t+T)
\alpha(T)\cos(k\pi e^t\alpha(T))e^{-T}dT.
\end{gathered}\end{equation}
Using similar arguments as above and the fact that $|\alpha(T)|\leq|\alpha(\log T_\pm(t))|=e^{-t}$ on the interval of integration in $T$ , 
the last term in \eqref{plt} is a bounded operator on $L^2(e^{-t})$, with norm bounded by $Ch^{-1}$.
Now for the first term of \eqref{plt}, it suffices to integrate by parts in $T$  
and use the fact that $\alpha(\pm h)=0$ to obtain
\[\begin{gathered}
(K_{h,k}\pl_tf)(t)=K_{h,k}f(t)\\
-\ell|B_h|^{-1}\int_{-h}^{\log T_-(t)}+\int_{\log T_+(t)}^h f(t+T)
(\pl_T\alpha)(T)\cos(k\pi e^t\alpha(T))e^{-T}dT.
\end{gathered}\]
If we cut-off to the region $e^t\sinh(h)\leq L$, this is an operator bounded on $L^2(e^{-t}dt)$ with norm bounded by
\[Ch^{-2}\int_{-h}^h 
|\pl_T\alpha(T)|dT=O(h^{-1})\]
where we used that $\alpha(T)$ is monotone on each of the $2$ intervals $[-h,0]$ and $[0,h]$ and that
its maximum is $\alpha(0)=O(h)$. 
Finally, the case $k=0$ is dealt with in the same way: the boundary terms in the integrals $(K^1_{h,0}+K^2_{h,0})f(t)$ 
cancel out those of $K_{h,0}^3f(t)$ and the other terms are estimated exactly like we did for $k\not=0$. 
This finishes the proof for the region $\{e^t\sinh(h)\geq \ell/2\}$.
As for the region ${ e^{t_0}\leq e^t\sinh(h)\leq \ell/2}$, we consider the expression \eqref{souslecusp} and apply the same 
exact method, this is even simpler.
\end{proof}
Then we end the proof of the Proposition with the 
\begin{lemma}
Let $L\geq \ell/2$, then for any $f\in L^2$ supported in the region $\{ t\leq \log(L/\sinh(h))\}$, we have
\[||\pl_tK_hf||_{L^2(M,dv_g)}\leq Ch^{-1}||f||_{L^2(M,dv_g)}.\]
\end{lemma} 
\begin{proof} 
We use the Fourier decomposition $f(t,y)=\sum_{k}f_k(t)e^{2i\pi k y/\ell}$ in the $\rr/\ell\zz$ variable and the expression of $K_h$ 
in Subsection \ref{firstexp}. We shall work on $L^2(\rr,dt)$ on each Fourier mode, which amounts
to conjugating by $e^{t/2}$ to pass unitarily from $L^2(e^{-t}dt)$ to $L^2(dt)$: 
let $\til{K}_h:=e^{t/2}K_he^{-t/2}$ and $\til{K}_{h,k}$ its decomposition on the $k$ Fourier mode $f_k(t)$ of $f(t,y)$. Then 
 from \eqref{expressionK^h_1} and similar arguments as for identity \eqref{plancherel}, we have 
\[ |B_h(t)|\til{K}_{h,k}f_k(t)= \int_{-\frac{e^{-t}\ell}{2}}^{\frac{e^{-t}\ell}{2}}e^{\frac{2\pi ikze^t}{\ell}}
\int e^{it\xi}\hat{f}_k(\xi)\sigma(z,\xi)d\xi dz\]
with 
\[\sigma(z,\xi):= \frac{(\cosh(h)+\sqrt{\sinh(h)^2-z^2})^{\demi+i\xi}-
(\cosh(h)-\sqrt{\sinh(h)^2-z^2})^{\demi+i\xi}}{(\demi+i\xi)(1+z^2)^{\demi+i\xi}}.\]
Then we obtain
\[ \begin{split}
\pl_t (|B_h(t)|\til{K}_{h,k}f_k)(t)= & \, \pl_t \int_{-\frac{e^{-t}\ell}{2}}^{\frac{e^{-t}\ell}{2}}e^{\frac{2\pi ikze^t}{\ell}}
\int e^{it\xi}\hat{f}_k(\xi)\sigma(z,\xi)d\xi dz\\
& = \, \int_{-\frac{e^{-t}\ell}{2}}^{\frac{e^{-t}\ell}{2}}e^{\frac{2\pi ikze^t}{\ell}}
\int e^{it\xi}\hat{f}_k(\xi)i\xi \sigma(z,\xi)d\xi dz \\
& \, + \int_{-\frac{e^{-t}\ell}{2}}^{\frac{e^{-t}\ell}{2}}\pl_z (e^{\frac{2\pi ikze^t}{\ell}})
\int e^{it\xi}\hat{f}_k(\xi) z\sigma(z,\xi)d\xi dz\\ 
& \,-\frac{(-1)^k e^{-t}\ell}{2}\int e^{it\xi}\hat{f}_k(\xi)( \sigma(\frac{e^{-t}\ell}{2},\xi)
+\sigma(-\frac{e^{-t}\ell}{2},\xi))d\xi.
\end{split}
\]
The term in the second line is clearly bounded by $Ce^{-t}||f_k||_{L^2(dt)}$ 
since $|\xi\sigma(z,\xi)|\leq C$ uniformly in $|z|\leq e^{-t}\ell/2$ and $k$. The same is true for 
the term in the last line while for the middle one, one can use integration by parts in $z$, which 
makes a boundary term of the same type as the last line term, plus a term similar to the first term but now with 
$\pl_z(z\sigma(z,\xi))$ instead of $\xi\sigma(z,\xi)$. Since $|\pl_z(z\sigma(z,\xi))|\leq C$ uniformly
in $|z|e^{-t}\ell/2$ and $k$, this achieves the proof. 
\end{proof}
The Proposition is then proved by combining the two Lemmas above.
\end{proof}

\end{document}